\documentclass[a4paper,11pt]{amsart}
\usepackage{amssymb}                                    % \rightrightarrows
\usepackage{mathrsfs}                    % \mathscr
\usepackage{amsmath}                    % Klammern, texte in Formeln
\usepackage{amsfonts}                   % Andere Schriftarten
\usepackage{amsthm}                     % Abschnitte, Theoreme
\usepackage[latin1]{inputenc}           % Umlaute
\usepackage[arrow, matrix, curve]{xy}    % Diagramme
\usepackage[english]{babel}                % Deutsche Sprache
\usepackage{bibgerm}				 % Deutsches BibTex
\usepackage{graphics}
\usepackage{bbm}
\usepackage[normalem]{ulem}
\usepackage{xcolor}

\usepackage{hyperref}
\theoremstyle{plain}
\newtheorem{thm}{Theorem}[section]
\newtheorem{corollary}[thm]{Corollary}
\newtheorem{proposition}[thm]{Proposition}
\newtheorem{lemma}[thm]{Lemma}
\newtheorem{example}[thm]{Example}
\newtheorem{remark}[thm]{Remark}
\newtheorem{prop}[thm]{Proposition}
\newtheorem{construction}[thm]{Construction}

\theoremstyle{definition}
\newtheorem{definition}[thm]{Definition}

\newcommand{\calS}{\mathcal{S}}
\newcommand{\Sp}{\mathrm{Sp}}
\newcommand{\SP}{\mathrm{Sp}}
\newcommand{\calO}{\mathcal{O}}
\newcommand{\calC}{\mathcal{C}}
\newcommand{\calD}{\mathcal{D}}
\newcommand{\calP}{\mathcal{P}}

\newcommand{\Op}{\mathcal{O}\mathrm{p}_\infty}
\newcommand{\Cat}{\mathcal{C}\mathrm{at}_\infty}
\newcommand{\Hom}{\mathrm{Hom}}
\newcommand{\imap}{\map}

\newcommand{\map}{\mathrm{map}}
\newcommand{\Fin}{\mathrm{Fin}_*}
\newcommand{\act}{\mathrm{act}}
\newcommand{\Add}{\mathrm{Add}}
\newcommand{\Pre}{\mathrm{Pre}}
\newcommand{\N}{\mathrm{N}}
\newcommand{\CMon}{\mathrm{CMon}} %_{\mathbb{E}_\infty}}
\newcommand{\CGrp}{\mathrm{CGrp}} %_{\mathbb{E}_\infty}}
\newcommand{\CAlg}{\mathrm{Alg}_{\mathbb{E}_\infty}}
\newcommand{\Map}{\mathrm{Map}}
\newcommand{\Fun}{\mathrm{Fun}}
\newcommand{\xto}{\xrightarrow}
\newcommand{\calK}{\mathcal{K}}
\newcommand{\Lex}{\mathrm{Lex}}
\newcommand{\id}{\mathrm{id}}
\newcommand{\pt}{\mathrm{pt}}
\newcommand{\Alg}{\mathrm{Alg}}
\newcommand{\rep}{\mathrm{rep}}
\newcommand{\lax}{\mathrm{lax}}
\newcommand{\St}{\mathrm{St}}
\newcommand{\fin}{\mathrm{fin}}

\newcommand{\psf}{\calS^\fin_*}
\newcommand{\Ind}{\mathrm{Ind}}
\newcommand{\Mul}{\mathrm{Mul}}
\newcommand{\op}{\mathrm{op}}

\DeclareMathOperator{\colim}{colim}

\newenvironment{psmallmatrix}
  {\left(\begin{smallmatrix}}
  {\end{smallmatrix}\right)}
\begin{document}

\title{Stable $\infty$-Operads and the multiplicative Yoneda Lemma}
\author{Thomas Nikolaus }
\date{\today}

\begin{abstract}
We construct for every $\infty$-operad  $\calO^\otimes$ with certain finite limits new $\infty$-operads 
%$\Sp(\calO)^\otimes$ 
of spectrum objects and %$\CGrp(\calO)^\otimes$ 
of commutative group objects in $\calO$. We show that  these are the universal stable resp. additive $\infty$-operads obtained from $\calO^\otimes$. We deduce that for a stably (resp. additively) symmetric monoidal $\infty$-category $\calC$ the Yoneda embedding factors through the $\infty$-category of exact, contravariant functors from $\calC$ to the $\infty$-category of spectra (resp. connective spectra) and admits a certain multiplicative refinement. 
As an application we prove that the identity functor $\Sp \to \Sp$ is initial among exact, lax symmetric monoidal endofunctors of the symmetric monoidal $\infty$-category $\Sp$ of spectra with smash product.
\end{abstract}

\maketitle

%Use of words checken: functor versus map of $\infty$-operads,
%operadically fully faithful
%$\Mul$
%TODO: Referenzen genauer machen.
%TODO: Checken ob ich immer $\infty$-operad anstatt operad geschrieben hab.
%tensor unit!

%%%%%%%%%%%%%%%%%%%%%%%%%%%%%%%%%%%%%%%%%%%%%%%%%%%%%%%%%%%%%%%%%%%%%%%%%%%%%%%%%%%
\section{Introduction}
%%%%%%%%%%%%%%%%%%%%%%%%%%%%%%%%%%%%%%%%%%%%%%%%%%%%%%%%%%%%%%%%%%%%%%%%%%%%%%%%%%%

Let $\calC$ be an $\infty$-category that admits finite limits. Then there is a new $\infty$-category 
$\Sp(\calC)$ of spectrum objects in $\calC$ that comes with a functor $\Omega^\infty: \Sp(\calC) \to \calC$.
It is shown in \cite[Corollary 1.4.2.23]{HA}  that this functors exhibits $\Sp(\calC)$ as the universal stable $\infty$-category obtained from $\calC$. This means that for every stable $\infty$-category $\calD$, post-composition with the functor $\Omega^\infty$ induces an equivalence
$$
 \Fun^{\Lex}\big(\calD, \Sp(\calC)\big) \to \Fun^{\Lex}\big(\calD, \calC\big),
$$
where $\Fun^{\Lex}$ denotes the $\infty$-category of finite limits preserving (a.k.a. left exact) functors.

The question that we want to address in this paper is the following. Suppose  $\calC$ admits a symmetric monoidal structure. Does  $\Sp(\calC)$ then also inherits some sort of 
 symmetric monoidal structure which satisfies a similar universal property in the world of symmetric monoidal $\infty$-categories? This is a question of high practical importance in applications in particular for the construction of algebra structures on mapping spectra and the answer that we give will be applied in future work by the author. \\
 
 Let us formulate the question more precisely. 
Assume that $\calC$ is a presentable and symmetric monoidal $\infty$-category such that the tensor bifunctor $\otimes: \calC \times \calC \to \calC$ preserves colimits in both variables separately. Then as a consequence of the results in \cite[Section 4.8.2] {HA} and \cite{GGN}  the $\infty$-category $\Sp(\calC)$ inherits a canonical closed symmetric monoidal structure  with respect to which the functor $\Omega^\infty$ becomes lax symmetric monoidal. The question is then whether for a stably symmetric monoidal $\infty$-category $\calD$ (i.e. $\calD$ is stable and the tensor bifunctor is exact in both variables seperately) the induced functor 
  \begin{equation}\label{one}
 \Fun^{\Lex}_\lax\big(\calD, \Sp(\calC) \big) \to \Fun^{\Lex}_\lax\big(\calD, \calC\big)
\end{equation}
 is an equivalence. Here $\Fun^\Lex_\lax$ denotes the $\infty$-category consisting of those lax symmetric monoidal functors whose underlying functor preserves finite limits.
 It is a consequence of the results of this paper that the functor \eqref{one} is indeed an equivalence of $\infty$-categories. 
 In fact our results are far more general. %Let us describe that in more detail:
 \begin{itemize}
 \item We introduce the concept of stability for $\infty$-operads $\calD$ (a.k.a. symmetric multicategories) generalizing the concept of stably symmetric monoidal $\infty$-categories (Definition \ref{defstable}). For every such stable $\infty$-operad $\calD$ the appropriate functor 
%\begin{equation}\label{two}
%\Alg^\Lex_{\calD}\left(\Sp(\calC)\right) \xrightarrow{\Omega^\infty} \Alg^\Lex_{\calD}\left(\calC\right)
%\end{equation}
 generalizing \eqref{one} is shown to be an equivalence (Corollary \ref{corimo}). This generalization makes sense since lax symmetric monoidal functors between symmetric monoidal $\infty$-categories are the same as maps between the underlying $\infty$-operads.
 %Here $\Alg^\Lex$ denotes the $\infty$-category of finite limit preserving operad maps (also called algebras) which are a direct generalization of lax symmetric monoidal functors.
 \item
 We prove that if $\calC$ is itself only an $\infty$-operad (which admits certain finite limits) then $\Sp(\calC)$ also inherits the structure of an $\infty$-operad which is stable (Theorem \ref{proppointad}). 
 This generalizes the symmetric monoidal structure on $\Sp(\calC)$ above (being symmetric monoidal is merely a \emph{property} of an $\infty$-operad).  
 Even if $\calC$ is symmetric monoidal, but not presentable, then it can happen that $\Sp(\calC)$ is only an $\infty$-operad and not symmetric monoidal. Again the generalisation of the  functor \eqref{one} is shown to be an equivalence (Corollary \ref{corimo}). 
 \item
 Conceptually the aforementioned results can be summarised by saying that  $\Sp(\calC)$ forms the cofree stable $\infty$-operad obtained from $\calC$. Following the pattern developed in \cite{GGN} we will generalize this to pointed and (pre)addi\-tive situations. More precisely we construct for every $\infty$-operad $\calC$ with sufficient limits new $\infty$-operads $\calC_*$ of pointed objects, 
 $\CMon(\calC)$ of commutative monoid objects and $\CGrp(\calC)$ of commutative group objects in  in $\calC$. We show that these are the cofree pointed, preadditive resp. additive  $\infty$-operads obtained from $\calC$ (see Section \ref{secbla}). 
 \end{itemize}
As a consequence of these abstract structural result we deduce that for every stably symmetric monoidal $\infty$-category $\calC$ 
the functor which sends an object $x \in \calC$ to the mapping spectrum $\map(\mathbbm{1},x)$
admits a lax symmetric monoidal structure. This implies that for every highly structured algebra $a$ in $\calC$ the mapping spectrum $\map(\mathbbm{1}, a)$ becomes a highly structured ring spectrum as well. Moreover the functor $\map(\mathbbm{1},-)$ is initial among all exact, lax symmetric monoidal functors from $\calC$ to the $\infty$-category of spectra (Corollary \ref{cor}). An immediate consequence is that the identity functor $\Sp \to \Sp$ is initial among exact, lax symmetric monoidal endofunctors of $\Sp$. This is, as tautological as it sounds, an important structural result for the symmetric monoidal $\infty$-category $\Sp$ which is not true, for example, for the $\infty$-category of chain complexes.

\subsection*{Relation to other work}
Lurie constructs in \cite[Section 6.2.4]{HA} a stabilization for a more restrictive subclass of $\infty$-operads than we do. His stabilization is equivalent to ours  and satisfies the same universal property but is closer to being symmetric monoidal (more precisely corepresentable). He however uses much more elaborate tools than we do, namely Goodwillie derivatives. We are not aware that a similar construction can be done in the (pre)additive setting, see also Remarks \ref{remarkde}  and \ref{Lurie}. 

In the paper \cite{GGN} it is shown that in the presentable setting the category $\Sp(\calC)$ satisfies the universal property that left adjoint, symmetric monoidal functors out of it are the same as functors of the same kind out of $\calC$. By passing to right adjoint functors this is equivalent to a universal property that is weaker but similar to the one here. However it only works in the presentable world. 
In \cite{MR3209352} Blumberg, Gepner and Tabuada take a similar presentable approach  to deduce that the lax symmetric monoidal functor induced by the tensor unit satisfies a certain universal property. We consider it to be one of the main contributions here that we do not need any presentability hypothesis and obtain even stronger forms of these statements.
Finally we make in this paper essential use of the Day convolution structures for $\infty$-categories as introduced by Glasman \cite{Glas} and further developed with Barwick \cite{2015arXiv150503098B}. We review and extend some of their results in Section \ref{day}.

\subsection*{Acknowledgements} This paper is part of a joint project with Markus Land. I want to thank him for several discussions and for helpful comments  on an early draft. I also thank David Gepner for comments on a draft.

%%%%%%%%%%%%%%%%%%%%%%%%%%%%%%%%%%%%%%%%%%%%%%%%%%%%%%%%%%%%%%%%%%%%%%%%%%%%%%%%%%%%%
\section{Operads and operadic limits}
%%%%%%%%%%%%%%%%%%%%%%%%%%%%%%%%%%%%%%%%%%%%%%%%%%%%%%%%%%%%%%%%%%%%%%%%%%%%%%%%%%%%%

We will freely use the language of $\infty$-categories as developed by Joyal, Lurie and others. Our main sources are the books by Lurie \cite{HTT,HA}. We will mostly follow the notation and terminology there. In particular 
for an $\infty$-category $\calC$ and objects $a,b \in \calC$ we denote the mapping space by $\Map_{\calC}(a,b)$ or only $\Map(a,b)$ if the ambient category $\calC$ is clear from the context.
The mapping space is considered as an object in the $\infty$-category $\calS$ of spaces. Several times in this paper, we will encounter the situation that for an $\infty$-category $\calC$ and objects $a,b \in \calC$ there is an `enriched' mapping object, which lies 
in another $\infty$-category $\calD$ with a functor $\calD \to \calS$ and which refines $\Map(a,b)$. In such a situation we will denote the enriched mapping object by $\map_{\calD}(a,b)$. 

We say that an inclusion of simplicial sets $\calC \subseteq \calD$ is a full simplicial subset if the $n$-simplices of $\calC$ are precisely
those $n$-simplices of $\calD$ for which all vertices lie in $\calC$. In particular a full simplicial subset can be specified by specifying a subset of the vertices.
The most important case is that $\calD$ is an $\infty$-category in which case it follows that every full simplicial subset is also an $\infty$-category and we say that it is a full subcategory.
For two $\infty$-categories $\calC$ and $\calD$ we denote by $\Fun(\calC, \calD)$ the functor 
category which is the simplicial set of maps from $\calC$ to $\calD$. We will add a superscript to indicate that we restrict to the full subcategory of functors with certain properties, for example
$\Fun^\Pi(\calC,\calD) \subseteq \Fun(\calC,\calD)$ for the product preserving functors and $\Fun^{\Lex}(\calC,\calD)  \subseteq \Fun(\calC,\calD)$ for the finite limit preserving (a.k.a. left exact) functors.  \\

Let $\Op$ denote the $\infty$-category of $\infty$-operads in the sense of \cite[Chapter 2]{HA}. These are really the $\infty$-categorical analogue of coloured, symmetric operads (a.k.a. symmetric multicategories)
but we will follow Lurie and simply call  them $\infty$-operads to avoid awkward language.
To be precise an $\infty$-operad is an inner fibration $\calO^\otimes \to \N\Fin$ that admits coCartesian lifts over inert morphisms in $\N\Fin$ such that the induced maps
$$
(\rho^1_!,...,\rho^n_!): \calO^\otimes_{\langle n \rangle} \to \calO^\otimes_{\langle 1 \rangle}  \times ... \times \calO^\otimes_{\langle 1 \rangle}
$$
are equivalences where $\rho^i: \langle n \rangle \to \langle 1 \rangle$ is the inert morphism in $\Fin$ with $\rho^i(i) = 1$ and $\rho^i(k) = 0$ for $k \neq i$. 
Here we denote the fibre of the fibration $\calO^\otimes \to \N\Fin$ over  the object 
$\langle n \rangle = \{0,...,n\} \in \N\Fin$ by $\calO^\otimes_{\langle n \rangle}$. 

For an $\infty$-operad $\calO^\otimes$ we will denote the underlying $\infty$-category 
as $\calO := \calO^\otimes_{\langle 1 \rangle}$.
By the usual abuse of notation we will usually only write $\calO^\otimes$ for the $\infty$-operad and leave the morphism $\calO^\otimes \to \N\Fin$ implicit. A map between $\infty$-operads $\calO^\otimes$ and ${\calO'}^\otimes$ is a functor $\calO^\otimes \to {\calO'}^\otimes$ over $\N\Fin$ which sends coCartesian lifts of inert morphisms to coCartesian lifts. 
For every pair of operads there is the $\infty$-category of operad maps $\calO \to \calO'$ which is denoted by $\Alg_{\calO}(\calO')$.

\begin{remark}
The notation $\Alg_{\calO}(\calO')$ comes from the fact that if $\calO^\otimes$ is a monochromatic (i.e. non-coloured) operad like $\mathbb{E}_n^\otimes$ for $0 \leq n \leq \infty$ then these $\infty$-categories really form the $\infty$-categorical analogue of algebras over classical operads.
 \end{remark}
 \begin{remark}
A symmetric monoidal $\infty$-category is a special case of an $\infty$-operad: it is an $\infty$-operad $\calC^\otimes$ such that the map $\calC^\otimes \to \N\Fin$ is coCartesian. 
If $\calC^\otimes$ and $\calC'^\otimes$ are symmetric monoidal $\infty$-categories then an operad map $\calC^\otimes \to \calC'^\otimes$
is the $\infty$-categorical analogue of a lax symmetric monoidal functor. It is potentially  misleading (thought mathematically consistent) to call such a functor a $\calC$-algebra in $\calC'$. In this situation we will denote the $\infty$-category $\Alg_{\calC}(\calC')$ also by 
$\Fun_\lax(\calC, \calC')$ and refer to it as lax symmetric monoidal functors. Such a lax symmetric monoidal functor $\calC^\otimes \to \calC'^\otimes$ is symmetric monoidal if it carries all coCartesian lifts to coCartesian lifts.
\end{remark}

Recall that a morphism in $\Fin$ is called active if only the basepoint maps to the basepoint, i.e. if it is obtained from a morphism of finite sets by adding a disjoint basepoint. For an operad $\mathcal{O}^\otimes$ we denote by 
$\calO^\otimes_\act$ the pullback of $\calO^\otimes$ along the inclusion $\N\Fin^\act \to \N\Fin$. In particular $\calO^\otimes_\act$ is an $\infty$-category and there is a fully faithful inclusion $i: \calO \subseteq \calO^\otimes_\act$ of $\infty$-categories.

\begin{definition}\label{deflimits}
Let $\calO^\otimes$ be an $\infty$-operad and $p: K^\lhd \to \calO$ be a cone in the underling $\infty$-category. We say that $p$ is an operadic limit if 
the induced cone $ip: K^\lhd \to \calO^\otimes_{\act}$ is a limit cone in the $\infty$-category $\calO^\otimes_\act$. 

Let $\calK$ be a collection of small simplicial sets. We say that an $\infty$-operad $\calO^\otimes$ admits all $\calK$-indexed operadic limits if for each functor $K \to \calO$ with $K \in \calK$ 
there exists an extension to an operadic limit cone.
By $\Op^\calK$ we denote the subcategory of $\Op$ consisting of $\infty$-operads that admit $\calK$-indexed operadic limits and operad maps that preserve $\calK$-index operadic limits.
Similary for $\calO^\otimes, \calO'^\otimes \in \Op^\calK$ we will denote by $\Alg^{\calK}_{\calO}(\calO') \subseteq \Alg_{\calO}(\calO')$ the full subcategory consisting of maps that preserve $\calK$-indexed operadic limits.

In the case $\calK = \{\emptyset\}$ we say that the operad admits a terminal object and write $\Op^{*}$ resp. $\Alg^*_{\calO}(\calO')$. % for $\Op^\calK$ (resp. $\Alg^{\calK}_{\calO}(\calO')$). 
Similarly for $\calK$ consisting of all finite discrete simplicial sets (including the empty set) we say that 
$\calO^\otimes$ admits operadic products and write $\Op^\Pi$ resp. $\Alg^{\Pi}_{\calO}(\calO')$ . Finally for $\calK$ consisting of all finite simplicial sets we say that $\calO^\otimes$ admits all finite operadic limits 
and write $\Op^{\Lex}$ resp. $\Alg^{\Lex}_{\calO}(\calO')$. 
\end{definition}
\begin{remark}
One can spell out the definition of an operadic limit more explicitly using the multi-mapping-spaces in an $\infty$-operad. For a sequence of objects $z_1,\ldots,z_n$ and $x$ in $\calO$ we will denote these 
multi-mapping-spaces  by $\Mul_{\calO}(z_1,\ldots ,z_n; x)$. Concretely these are the mapping spaces 
$$
\Mul_{\calO}(z_1,\ldots,z_n; x) = \Map_{\calO^\otimes_{act}}(z_1 \boxtimes \ldots \boxtimes z_n; x)
$$
where we consider the list $(z_1, \ldots, z_n)$ as an object denoted $z_1 \boxtimes \ldots \boxtimes z_n$ in $\calO^\otimes_\act$. For example the mapping spaces $\Mul_{\calO}(z,x)$in the $\infty$-operad $\calO^\otimes$ agrees with the mapping space $\Map_{\calO}(z,x)$ in the underlying $\infty$-category. 
% $$
% \xymatrix{
% \Map_{\calO}(z_1,...,z_n; x) \ar[r]\ar[d] & (\calO^\otimes_\act)^{\Delta^1} \ar[d] \\
% \Delta^0 \ar[r]^-{(z_1,...,z_n; x)} & \calO^\otimes_\act \times \calO^\otimes_\act
% }
% $$
%One can show that these multi-hom spaces form a functor $(\prod_n \calO^{op}) \times \calO \to \calS$
%where $\calS$ denotes the $\infty$-category of spaces.

Then a cone $p: K^{\lhd} \to \calO$ is an operadic limit if and only if for every sequence of objects $z_1, \ldots ,z_n$ 
in $\calO$ the induced cone
$$
\Mul_{\calO}(z_1, \ldots, z_n; p): K^\lhd \to \calS
$$
is a limit in the $\infty$-category $\calS$ of spaces. In particular an operadic limit is also a limit in the underlying $\infty$-category $\calO$. 
\end{remark}
\begin{example}\label{exym}
Recall that a symmetric monoidal $\infty$-category is by definition an  $\infty$-operad such that the map $\calC^\otimes \to \N \Fin$ is coCartesian. Then a cone $\calK^\lhd\to \calC$ is an operadic limit if and only if 
it is a limit in the underlying $\infty$-category $\calC$. This can be easily seen using the description in terms of multi-mapping spaces and the fact that in this case $\Mul_{\calC}(z_1, \ldots ,z_n; -)$ is corepresentable 
by $z_1 \otimes ....\otimes z_n$. In particular $\calC^\otimes$ admits all $\calK$-indexed operadic limits precisely if the underlying $\infty$-category $\calC$ admits all $\calK$-indexed limits.  
\end{example}

\begin{example}
It is not true in general that an $\infty$-operad admits $\calK$-shaped limits if the underlying $\infty$-category does. Consider for example the $\infty$-category $\Delta^0$ considered as an $\infty$-operad without higher operations. Concretly the $\infty$-operad is given by $\N\Fin^{\mathrm{inert}} \to \N\Fin$ where $\Fin^\mathrm{inert} \subseteq \Fin$ is the subcategory consisting of all finite pointed sets and  inert morphism. This $\infty$-operad is also called the trivial operad. Then the unique object is terminal in the underlying $\infty$-category but not operadically terminal.
\end{example}

We now want to give a `converse' to Example \ref{exym}. 
We say that a map of $\infty$-operads $\calO^\otimes$ and ${\calO'}^\otimes$ 
is fully faithful if it is fully faithful as a functor $\calO^\otimes \to {\calO'}^\otimes$ between $\infty$-categories. Equivalently if for every sequence of objects $z_1,...,z_n,x$ in 
$\calO$ the induced morphism 
$$
\Mul_{\calO}(z_1,\ldots ,z_n; x) \to \Mul_{{\calO'}}\big(F(z_1),\ldots ,F(z_n); F(x) \big)
$$
is a homotopy equivalence. For example a symmetric monoidal functor $F: \calC^\otimes \to \calC'^\otimes$ between symmetric monoidal $\infty$-categories $\calC^\otimes$ and $\calC'^\otimes$ is fully faithful as a map of $\infty$-operads precisely 
if the underlying functor $F_{\langle 1 \rangle}: \calC \to \calC'$ between $\infty$-categories is fully faithful.

A large symmetric monoidal $\infty$-category $\calC^\otimes$ is called presentably symmetric mon\-oi\-dal if the underlying $\infty$-category $\calC$ is presentable and the tensor bifunctor 
$\otimes: \calC \times \calC \to \calC$ preserves colimits separately in each variable. 

\begin{proposition}\label{limits}
For a small $\infty$-operad $\calO^\otimes$ and a collection of finite simplicial sets $\calK$ the following are equivalent
\begin{enumerate}
 \item $\calO^\otimes$ admits operadic limits of shape $\calK$
 \item $\calO^ \otimes$ admits a fully faithful inclusion $i: \calO^\otimes \to \calC^\otimes$ where $\calC^\otimes$ is presentably symmetric monoidal 
and such that the essential image of $i(\calO) \subseteq \calC $ is closed under $\calK$-indexed limits.
 \item $\calO^ \otimes$ admits a fully faithful inclusion $i: \calO^\otimes \to \calC^\otimes$ where $\calC^\otimes$ is small, symmetric monoidal, $\calC$ admits  
$\calK$-index limits and the essential image  $i(\calO) \subseteq \calC$ is closed under $\calK$-indexed limits.
\end{enumerate}
\end{proposition}
\begin{proof}
Recall that for an $\infty$-operad we have the inclusion $i: \calO \subseteq \calO^\otimes_\act$. Moreover $\calO^\otimes_\act$ admits a symmetric monoidal structure whose associated tensor bifunctor will be denoted by
$\boxtimes: \calO^\otimes_\act \times \calO^\otimes_\act \to \calO^\otimes_\act$. Note that this functor is denoted as $\oplus$ in \cite[Remark 2.2.4.6]{HA}. But since this could lead to confusion with biproducts in the additive situation we prefer to use $\boxtimes$.

To show $(1) \Rightarrow (2)$ assume that $\calO^\otimes$ admits $\calK$-indexed operadic limits. It is shown in \cite[Section 2.2.4]{HA} that the fully faithful inclusion $\calO\to \calO^\otimes_{\act}$ of $\infty$-categories admits an extension to a
fully faithful map of $\infty$-operads 
$\calO^\otimes \to (\calO^\otimes_{\act})^\boxtimes$. %where $(\calO^\otimes_\act)^\boxtimes$ is a symmetric monoidal $\infty$-category with underlying $\infty$-category $\calO^\otimes_\act$. This is in fact the universal 
%symmetric monoidal envelope of $\calO^\otimes$. 
Using \cite[Corollary 4.8.1.12.]{HA} we can equip the $\infty$-category $\calP(\calO^\otimes_\act) = \Fun\left((\calO^\otimes_\act)^{op}, \calS\right)$ of space valued presheaves on $\calO^\otimes_\act$ with a symmetric monoidal structure
such that
the Yoneda embedding $\calO^\otimes_\act \to \calP(\calO^\otimes_\act)$ admits a symmetric monoidal refinement. This structure is in fact equivalent to the Day convolution that we will discuss in Section \ref{day} but we will not need this equivalence here.

Then by the Yoneda Lemma the induced map of $\infty$-operads 
$$(\calO^\otimes_{\act})^\boxtimes \to \calP(\calO^\otimes_\act)^\otimes $$
is  fully faithful. But by construction $\calP(\calO^\otimes_\act)^\otimes$ is presentably symmetric monoidal and the composition
$
\calO^\otimes \to \calP(\calO^\otimes_\act)^\otimes
$
is  fully faithful. Also the underlying functor $\calO\to \calP(\calO^\otimes_\act)$ preserves $\calK$-indexed operadic limits since the map $\calO^\otimes \to \calO^\otimes_\act$ does so by definition 
of operadic limits and the Yoneda embedding preserves all limits. Thus it also preserves all operadic limits (which exist in the target).

For $(2) \Rightarrow (3)$ we use that the inclusion $\calO \to \calC$ factors through the $\kappa$-compact objects $\calC^\kappa$ for some cardinal $\kappa$ since 
$\calO$ is small.  Enlarging $\kappa$ if necessary we can assume that $\calC^\kappa$ is closed under all finite limits by \cite[5.4.7.4]{HTT}. 
Enlarging $\kappa$ further we can also  assume that $\calC^\kappa$ is closed under tensor products. This then shows the claim.

For $(3) \Rightarrow (1)$ assume the existence of an embedding $j: \calO^\otimes \to \calC^\otimes$ as stated in $(3)$. Then for every functor $K \to \calO^\otimes$ the limit  again lies in the essential image of 
$j_{\langle 1 \rangle}$. Since it is an operadic limit in $\calC^\otimes$ it follows from fully faithfulness that it is also an operadic limit in $\calO^\otimes.$
\end{proof}
\begin{remark}\label{remarksize}
We assume implicitly throughout that paper that our $\infty$-operads $\calO^\otimes$ are small. The presentably symmetric monoidal category $\calC^\otimes$ that occurred in the last proposition however is necessarily large (i.e. the simplicial set $\calC^\otimes$ is a simplicial set in the next universe).  But in practice the $\infty$-operads $\calO^\otimes$ that we care about are often already large. In that case the symmetric monoidal category $\calC^\otimes$ will be very large i.e. lies in the next universe after large. Everything then works logically exactly the same as now (in fact large sets also satisfy the ZFC axioms). Thus for simplicity of language we will keep using the implicit assumption that all $\infty$-operads $\calO^\otimes$ are small unless otherwise stated but note that all results also hold for $\infty$-operads in a larger universe with the changes that those which are now large have to lie in the successor universe. 
\end{remark}

\begin{definition}
An operadic adjunction between $\infty$-operads $p\colon \calO^\otimes \to  \N\Fin$ and $p'\colon \calO'^\otimes \to \N\Fin$ consists of operad maps $F: \calO^\otimes \to {\calO'}^\otimes$ and $G: \calO'^\otimes \to \calO^\otimes$ 
together with a  transformation 
$\alpha: \id_{\calO^\otimes} \to GF$
that exhibits $G$ as a right adjoint to $F$ and such that $p(\alpha)$ is the identity transformation from $p$ to itself. In this case we say that $F$ is operadically left adjoint to $G$ and $G$ is operadically right adjoint to $F$.
\end{definition}

Clearly the operadic right adjoint is, if it exists, essentially uniquely determined by the left adjoint and vice versa. 

\begin{remark}
Note that this notion is a special case of the notion of relative adjunctions discussed by Lurie in \cite[Section 7.3.2]{HA}.  One can translate the concept of an operadic adjunction into the language of multi-mapping-spaces.
Then an operadic adjunction consists of two operad  maps $F$ and $G$ together with a transformation $\alpha: \id \to GF$ such that for every sequence of objects $x_1,...,x_n \in \calO$ and every object $y \in \calO'$ the 
induced transformation 
$$
\Mul_{\calO'}\big(F(x_1),\ldots,F(x_n); y\big)  \xto{\alpha^* \circ G_{(x_1,\ldots,x_n; y)}}  \Mul_{\calO}\big(x_1,\ldots, x_n; G(y)\big) 
$$ 
is a homotopy equivalence of spaces. 
\end{remark}

The next lemma in particular shows that an operadic adjunction always has an underlying ordinary adjunction. In other words: the underlying functor of the operadic right adjoint is the right adjoint of the underlying functor.
\begin{lemma}
For an operadic adjunction we obtain by restriction induced adjunctions
$$
F_{\langle 1 \rangle} : \calO \leftrightarrows \calO': G_{\langle 1 \rangle}
\qquad \text{and} \qquad 
F_{\act} : \calO^\otimes_\act \leftrightarrows \calO'^\otimes_\act: G_{\act}
$$ 

\end{lemma}
\begin{proof}
This follows directly from \cite[Proposition 7.3.2.5]{HA}.
\end{proof}
It easily follows from the second equivalence of this lemma that an operadic right adjoint functor preserves operadic limits.

\begin{example}\label{exampleeins}
\begin{enumerate}
\item
Let $\calC^\otimes$ and $\calC'^\otimes$ be symmetric monoidal $\infty$-categories and $F: \calC^\otimes \to \calC'^\otimes$ be a map of operads. 
Then $F$ admits an operadic right adjoint if and only if the 
underlying functor $F_{\langle 1 \rangle}: \calC \to \calC'$ admits a right adjoint. This follows from \cite[Corollary 7.3.2.7]{HA}.
\item \label{pointeuzwei}
Assume that $\calD \subseteq \calC$ is a full, reflective subcategory of the $\infty$-category $\calC$ with localization $L: \calC \to \calD$. Assume that $\calC$ has a symmetric monoidal struture $\calC^\otimes$. Then we have the induced operad structure $
i: \calD^\otimes \subseteq \calC^\otimes$ such that the inclusion $i$ is a fully faithful map of $\infty$-operads. Assume that the localization is compatible with the symmetric monoidal structure in the following sense:
\begin{quote}
For every $L$-equivalence $X \to Y$ in $\calC$ and every object $Z \in \calC$ the induced map $X \otimes Z \to Y \otimes Z$ is also an $L$-equivalence.
\end{quote}
Then $\calD^\otimes$ is symmetric monoidal, and the inclusion $i$ admits an operadic left adjoint $L^\otimes$ which is a symmetric monoidal functor. This is \cite[Proposition 2.2.1.9.]{HA}.
\item \label{pointudrei}
In the the situation of \eqref{pointeuzwei} assume that $\calC$ admits internal mapping objects. That is a functor $\imap_\calC(X,-): \calC \to \calC$ right adjoint to $- \otimes X  : \calC \to \calC$. 
Then instead of the criterion given in \eqref{pointeuzwei} it suffices to check the following
\begin{quote}
For every object $X \in \calD$ and every object $Z \in \calC$ the internal mapping object $\imap_\calC(Z,X) \in \calC$ is equivalent to an object in $\calD$.
\end{quote}
The equivalence of the two criteria follows easily from the adjunction property of the tensor product and the internal mapping space.
\end{enumerate}
\end{example}

%%%%%%%%%%%%%%%%%%%%%%%%%%%%%%%%%%%%%%%%%%%%%%%%%%%%%%%%%%%%%%%%%%%%%%%%%%%%%%%%%%%%%
\section{Day convolution of \texorpdfstring{$\infty$}{infinity}-operads}\label{day}
%%%%%%%%%%%%%%%%%%%%%%%%%%%%%%%%%%%%%%%%%%%%%%%%%%%%%%%%%%%%%%%%%%%%%%%%%%%%%%%%%%%%

In this section we recall and expand some results about the Day convolution of $\infty$-categories as developed by Glasman \cite{Glas}.
This Day convolution for $\infty$-categories is a generalization of the classical one for ordinary categories by Day \cite{day}. 
 The results of this section will be needed in some of the proofs in later chapters. But we have tried to formulate the statements in later chapters  without reference to the terminology introduced here so that this section can be skipped at a first reading. \\

Assume $\calC^\otimes \to \N\Fin$ is a symmetric monoidal $\infty$-category, i.e. a coCartesian fibration. Assume moreover that $\calO^\otimes \to \N\Fin$ is an $\infty$-operad.
We consider the internal hom in simplicial sets  over $\N\Fin$ denoted by $\hom_{{/\N\Fin}}\left(\calC^\otimes, \calO^\otimes\right)$. 
This is by definition the simplicial set over $\N\Fin$ (unique up to isomorphism) such that for every other simplicial set $K \to \N\Fin$ there is a bijection 
$$\Hom\left(K, \hom_{{/\N\Fin}}\left(\calC^\otimes, \calO^\otimes\right)\right) \xto{\sim} \Hom\left(K \times_{\N\Fin} \calC^\otimes, \calO^\otimes\right)$$
which naturally depends on $K$. Here $\Hom$ denotes the set of homomorphisms of simplicial sets over $\N\Fin$. In particular the fibre of  $\hom_{{/\N\Fin}}\left(\calC^\otimes, \calO^\otimes\right)$ over $\langle n \rangle$ is isomorphic to the simplicial set 
$\Fun\left(\calC^\otimes_{\langle n \rangle},\calO^\otimes_{\langle n \rangle}\right)$.
% 
% between morphismsm $K \to \hom_{{/\N\Fin}}\left(\calC^\otimes, \calO^\otimes\right)$
% over $\N\Fin$ and morphisms $K \times_{\N\Fin} \calC^\otimes \to \calO^\otimes$ over $\N\Fin$.

\begin{lemma}\label{inner}
The morphism $\hom_{{/\N\Fin}}\left(\calC^\otimes, \calO^\otimes\right) \to \N\Fin$ is an inner fibration.
\end{lemma}
\begin{proof}
This follows as in \cite{Glas}: we have to find a lift in every diagram
$$
\xymatrix{
\Lambda^n_i \ar[r]\ar[d] &\hom_{{/\N\Fin}}\left(\calC^\otimes, \calO^\otimes\right)\ar[d] \\
\Delta^n \ar[r]\ar@{-->}[ru] & \N\Fin
}
$$
for $0 < i < n$. But this is by adjunction the same as a lift in the diagram
$$
\xymatrix{
\Lambda^n_i \times_{\N\Fin} \calC^\otimes \ar[r]\ar[d] & \calO^\otimes\ar[d] \\
\Delta^n \times_{\N\Fin} \calC^\otimes \ar[r]\ar@{-->}[ru] & \N\Fin
}
$$
Since $\calO^\otimes \to \N\Fin$ is inner it suffices to show that the left vertical morphism is inner anodyne. But that follows from the dual of \cite[Proposition 3.3.1.3]{HTT}.
\end{proof}

\begin{definition}\label{definitionday}
Let $\calC^\otimes$ be a symmetric monoidal $\infty$-category and $\calO^\otimes$ be an $\infty$-operad. Then the Day convolution is defined as the the full simplicial subset
$
\Fun(\calC, \calO)^\otimes \subseteq \hom_{{/\N\Fin}}\left(\calC^\otimes, \calO^\otimes\right)
$
 whose fibre over $\langle n \rangle$ consists of those functors 
$$F: \calC^\otimes_{\langle n \rangle} \to \calO^\otimes_{\langle n \rangle}$$
 such that under the equivalences 
 $$
 \rho: \calC^\otimes_{\langle n \rangle} \xto{\sim} \prod_n \calC \qquad \rho': \calO^\otimes_{\langle n \rangle} \xto{\sim} \prod_n \calO
 $$
 we have an equivalence of functors $\rho' \circ F \simeq (F_1 \times...\times F_n) \circ \rho$ for functors $F_i: \calC\to \calO$. 
\end{definition}
The last condition can be expressed more informally by saying that the functor $F$ is in diagonal form.
Note that the notation here is slightly abusive since $\Fun(\calC,\calO)^\otimes$ depends on the operad structures $\calC^\otimes$ and $\calO^\otimes$ and not only on the underlying $\infty$-categories $\calC$ and $\calO$. But the underlying $\infty$-category 
$\Fun(\calC,\calO)^\otimes_{\langle 1 \rangle}$ is the functor category $\Fun(\calC,\calO)$ and therefore we choose the notation as it is.

\begin{proposition}\label{propopeee}
Assume $\calC^\otimes$ is a small symmetric monoidal $\infty$-category and $\calO^\otimes$ is a small $\infty$-operad $\calO$. Then  $\Fun(\calC, \calO)^\otimes \to \N\Fin$ is an $\infty$-operad. %If $\calD$ is symmetric monoidal and admits colimits, then $\Fun(\calC,\calD)^\otimes$ is symmetric monoidal as well. 
If $\calD$ is presentably symmetric monoidal, then so is $\Fun(\calC,\calD)^\otimes$.

There is an equivalence of symmetric monoidal $\infty$-categories $\CAlg(\Fun(\calC,\calO)) \simeq \Alg_\calC(\calO)$. In the case that $\calO$ is symmetric monoidal this means that a commutative algebra in $\Fun(\calC,\calO)$ is essentially the same thing as a lax symmetric monoidal
functor $\calC \to \calO$. 
\end{proposition}
\begin{proof}
Assume first that $\calO^\otimes$ is presentably symmetric monoidal. Then the underlying $\infty$-category $\Fun(\calC,\calO)$ of $\Fun(\calC, \calO)^\otimes$ is presentable. The fact that it is symmetric monoidal is \cite[Proposition 2.9]{Glas} and the fact that the tensor product preserves
colimits separately in each variable is \cite[Proposition 2.11]{Glas}. Now let $\calO$ be a small $\infty$-operad. Choose an operadically fully faithful embedding $i: \calO^\otimes \to \calD^\otimes$
 where $\calD$ is presentably symmetric monoidal. To see that this is possible apply Proposition \ref{limits} with $\calK = \emptyset$.
 Let $i(\calO)^\otimes \subseteq \calD^\otimes$ be the full subcategory spanned by the essential image of $i$. Then $\Fun(\calC, i(\calO))^\otimes \subseteq \Fun(\calC, \calD)^\otimes$ is also a full subcategory. Since it is 
spanned by the products of elements in $\Fun(\calC, \calO)$ it follows immediately that it is again an $\infty$-operad. 
Now finally there is an evident equivalence of $\infty$-operads  $\calO^\otimes \to i(\calO)^\otimes$. This induces an equivalence of the $\infty$-categories $\Fun(\calC,\calO)^\otimes \simeq \Fun(\calC,i(\calO))^\otimes$ over $\N\Fin$. Since according to \ref{inner}
the first is an inner fibration it follows that it is an $\infty$-operad from the fact that the second is.

The second claim directly follows from \cite[Proposition 2.12.]{Glas}.
\end{proof}

\begin{remark}
The conclusion of Proposition \ref{propopeee} is true more generally than stated (see \cite{Glas} and \cite{2015arXiv150503098B}): first one can allow the source $\calC^\otimes$ to be pro-op-monoidal instead of symmetric mon\-oi\-dal. That means that $\calC^\otimes$ is an $\infty$-operad such that 
$\calC^\otimes \to \N\Fin$ is a flat inner fibration (see 
\cite[Definition B.3.8]{HA}). In that case $\Fun(\calC,\calO)^\otimes$ defined with the same formula is still is an $\infty$-operad. Also for $\Fun(\calC, \calO)^\otimes$ to be symmetric monoidal
one does not need $\calO^\otimes$ to be presentably symmetric monoidal, but it is enough if $\calO^\otimes$ is symmetric monoidal and has sufficient colimits. We shall however not need these generalizations in this paper.
\end{remark}

In order to understand the coCartesian lifts in $\Fun(\calC, \calO)^\otimes$ let us unfold what a 1-simplex in $\hom_{{/\N\Fin}}\left(\calC^\otimes, \calO^\otimes\right)$ is. First  it covers a morphism 
$f: \langle n \rangle \to \langle m \rangle$ in $\Fin$. We denote by $\calC^\otimes_f$ the pullback of $\calC^\otimes \to \N\Fin$ along the functor  $\Delta^1 \to \N\Fin$ induced by $f$. Similarly we have $\calO^\otimes_f$. 
Then by the universal property a 1-simplex  in  $\hom_{{/\N\Fin}}\left(\calC^\otimes, \calO^\otimes\right)$ covering $f$ is the same as a morphism $\calC^\otimes_f \to \calO^\otimes_f$ over $\Delta^1$. 
Now let us assume that $\calO^\otimes$ and $\calC^\otimes$ admit coCartesian lifts for $f$ (this is by assumption the case for $\calC$ and for $\calO$ it is the case if $f$ is inert or if $\calO$ is symmetric monoidal). We denote the resulting functors
by 
$$
f_!^\calC: \calC^\otimes_{\langle n \rangle} \to \calC^\otimes_{\langle m \rangle} \qquad \text{and} \qquad f_!^\calO: \calO^\otimes_{\langle n \rangle} \to \calO^\otimes_{\langle m \rangle}
$$ 
Then a morphism $\calC^\otimes_f \to \calO^\otimes_f$ is essentially the same thing as a pair of functors $F: \calC^\otimes_{\langle n \rangle} \to \calO^\otimes_{\langle n \rangle}$ and 
$G: \calC^\otimes_{\langle m \rangle} \to \calO^\otimes_{\langle m \rangle}$ together with a transformation $\eta: f^\calO_! \circ  F \to G \circ f^\calC_!$. 
% We can depict that as a square 
% $$
% \xymatrix{
% \calC^\otimes_{\langle n \rangle} \ar[r]^{F} \ar[d]^{f_!^\calC} & \calO^\otimes_{\langle n \rangle}\ar[d]^{f_!^\calO} \\
% \calC^\otimes_{\langle m \rangle} \ar[r]^{G}  & \calO^\otimes_{\langle m \rangle} 
% }
% $$
\begin{lemma}\label{lemalifts}
Let $f: \langle n \rangle \to \langle m \rangle $ be a morphism in $\Fin$.  Now assume that either $\calO^\otimes$ is presentably symmetric monoidal or that $f$ is inert (also recall that $\calC^\otimes$ is symmetric monoidal).
An edge in $\Fun(\calC, \calO)^{\otimes}$ over $f$ is essentially the same thing as a (not necessarily commutative) square
$$
\xymatrix{
\calC^\otimes_{\langle n \rangle} \ar[rr]^{F_1 \times ... \times F_n} \ar[d]^{f_!^\calC} && \calO^\otimes_{\langle n \rangle}\ar[d]^{f_!^\calO} \\
\calC^\otimes_{\langle m \rangle} \ar[rr]^{G_1 \times ... \times G_m}  && \calO^\otimes_{\langle m \rangle} 
}
$$
together with a transformation $\eta: f^\calO_! \circ  F \to G \circ f^\calC_!$ where we abbreviate $F = F_1 \times ... \times F_n$ and $G = G_1 \times ... \times G_m$. 
This edge is coCartesian if and only if the transformation $\eta$ exhibits $G$ as the left Kan extension of 
$f_!^\calO \circ F$ along $f_!^\calC$. 
\end{lemma}
\begin{proof}
The first part follows from the discussion above and the second is a reformulation of the description of coCartesian lifts given in \cite{Glas}.
\end{proof}

\begin{corollary}\label{univdaybla}
Let $\calC^\otimes$ be symmetric monoidal and $\calD^\otimes$ be presentably symmetric monoidal. Then we have a natural equivalence
$$
\Map_{\Fun(\calC,\calD)}(E \otimes F, G) \simeq  \Map_{\Fun(\calC \times \calC, \calD)}( \otimes^{\calD} \circ (E \times F), G \circ \otimes^\calC)  
$$
where $\otimes^\calC: \calC \times \calC \to \calC$ and $\otimes^\calD: \calD \times \calD \to \calD$ are the tensor bifunctors of $\calC$ and $\calD$.
\end{corollary}
\begin{proof}
Note that in the terminology of Lemma \ref{lemalifts}
the tensor bifunctors are given by $f^\calC_!$ and $f^\calD_!$ for $f$ the unique active morphisms $\langle 2 \rangle \to \langle 1 \rangle$. Since the tensor product of $E$ and $F$ is the left Kan extension of $\otimes^{\calD} \circ (E \times F)$
along $\otimes^\calC$ the claim follows from the universal property of left Kan extension.
\end{proof}

Note that the construction of $\Fun(\calC,\calO)^\otimes$ as a subset of $\hom_{{/\N\Fin}}\left(\calC^\otimes, \calO^\otimes\right)$ 
allows us to immediately deduce certain functorialities. It is on the point set level (i.e. as simplicial sets over $\N\Fin$) contravariantly functorial in the first variable and covariantly in the second. We can now show that these functorialities are compatible with the operad structures.

\begin{corollary}\label{morphism_co}
Let $\calC^\otimes$ be a symmetric monoidal $\infty$-category. 
For every morphism $p: \calO^\otimes \to \calO'^\otimes$ of $\infty$-operads the  morphism $p_*: \Fun(\calC,\calO)^\otimes \to \Fun(\calC,\calO')^\otimes$ induced by post-composition with $p$ is a morphism of $\infty$-operads.
If $\calO$ and $\calO'$ are presentably symmetric monoidal and $p$ is a symmetric monoidal functor that preserves colimits then $p_*$ is also a symmetric monoidal functor.
\end{corollary}
\begin{proof}
The only thing that we have to prove is that $p_*$ preserves certain coCartesian lifts, namely those over intert morphisms in $\N\Fin$ in the first case and all coCartesian lifts in the presentably symmetric monoidal case. Thus let let $f$ be a morphism in $\N\Fin$ as in Lemma \ref{lemalifts}. Then for a lift of $f$ we obtain a diagram 
$$
\xymatrix{
\calC^\otimes_{\langle n \rangle} \ar[rr]^{F_1 \times ... \times F_n} \ar[d]^{f_!^\calC} && \calO^\otimes_{\langle n \rangle}\ar[d]^{f_!^\calO} \ar@{=>}[lld]_\eta \ar[rr]^{p_{\langle n \rangle}} &&  \calO'^\otimes_{\langle n \rangle}\ar[d]^{f_!^{\calO'}}  \\
\calC^\otimes_{\langle m \rangle} \ar[rr]^{G_1 \times ... \times G_m}  && \calO^\otimes_{\langle m \rangle} \ar[rr]^{p_{\langle m \rangle}} && {\calO'}^\otimes_{\langle m \rangle}
}
$$
in which the right hand side commutes (in sense that there is a filler cell that is invertible and which is therefore suppressed in the notation) since $p$ 
preserves coCartesian lifts over $f$.
Using the description of coCartesian lifts given in Lemma \ref{lemalifts} we have to show  that the composition cell $p_{\langle m \rangle} \circ \eta$  exhibits the composition $p_{\langle m \rangle} \circ G$ as the left Kan extension of $f_!^{\calO'} \circ p_{\langle n \rangle} \circ F \simeq  p_{\langle m \rangle}  \circ f_!^{\calO} \circ F $ along $f_!^{\calC}$. 

We now disntinguish two cases: first if the morphism $f$ is inert, then all the pushforward functors $f_!$ are product projections. Then $\eta$ exhibits a left Kan extensions precisely if $G$ is equivalent to the appropriate projection of the $F_i's$, i.e.  $G \simeq \prod F_{f(i)}$. In particular $\eta$ is an  equivalence in this case. Similar we see that $p_{\langle m\rangle}$ is also in product form. From this description it is immediate that $p_{\langle m \rangle} \circ \eta$ is an equivalence and exhibits a product projection, hence a left Kan extension. 

The second case that we have to consider is the $f$ is not necessarily inert but $p_{\langle 1 \rangle}$ and hence also $p_{\langle m \rangle}$ preserve colimits. But then it is clear that $p_{\langle m \rangle}$  also preserves left Kan extensions (using the pointwise formulas) and we are done.
\end{proof}

\begin{corollary}\label{morphism_contra}
Let  $q: \calC^\otimes \to \calC'^\otimes$ be an operad morphism between symmetric monoidal $\infty$-categories $\calC$ and $\calC'$, i.e. a lax symmetric monoidal functor. Then  the morphism $q^*: \Fun(\calC',\calO)^\otimes \to \Fun(\calC,\calO)^\otimes$ induced by pre-composition with $q$ is a morphism of $\infty$-operads. 
If $\calO^\otimes$ is presentably symmetric monoidal and $q$ is a symmetric monoidal functor, then $q^*$ admits an operadic left adjoint which is symmetric monoidal.
\end{corollary}
\begin{proof}
For the first part of the proof we proceed similarly as in the proof of Corollary \ref{morphism_co}. We have to show that $q^*$ preserves coCartesian lifts of inert morphisms.  Thus consider a coCartesian lift in $\Fun(\calC',\calO)^\otimes$ over an inert morphism $f$ in $\Fin$. Following the discussion in the Proof of Corollary \ref{morphism_co} we have a diagram
$$
\xymatrix{
\calC^\otimes_{\langle n \rangle} \ar[rr]^{q_{\langle n \rangle}} \ar[d]^{f_!^{\calC}}&& \calC'^\otimes_{\langle n \rangle} \ar[rr]^{F_1 \times ... \times F_n} \ar[d]^{f_!^{\calC'}} && \calO^\otimes_{\langle n \rangle}\ar[d]^{f_!^\calO}  \\
\calC^\otimes_{\langle m \rangle} \ar[rr]^{q_{\langle m \rangle}}  && \calC'^\otimes_{\langle m \rangle} \ar[rr]^{G_1 \times ... \times G_m}  && \calO^\otimes_{\langle m \rangle}
}
$$
where we do not draw $\eta$ as a filler in the right hand diagram since it is invertible. Then again $q_{\langle n \rangle}$ is in product form as well and as in the proof of Corollary \ref{morphism_co} before we see that the outer diagram is also a left Kan extension.

Thus it remains to show the second part of the claim. We want to use \cite[Corollary 7.3.2.12]{HA}. Thus we first have to show that the underlying functor $q^*_{\langle 1 \rangle}:  \Fun(\calC',\calO)\to \Fun(\calC,\calO)$ admits a left adjoint, but this is clear since it preserves all limits and colimits and since $\calO$ is presentable and therefore the functor categories are presentable as well. This left adjoint is given by left Kan extension and we denote this left adjoint by $q_!$. Then we have to verify condition (2) of \cite[Corollary 7.3.2.12]{HA} which in this case comes down to  the following: % for every $\langle n \rangle \in \Fin$ 
the evident natural transformations
$$
q_! \circ \otimes^{\Fun(\calC,\calO)} \to \otimes^{\Fun(\calC',\calO)} \circ (q_! \times q_!) \qquad \text{and} \qquad q_!(\mathbbm{1}_{\Fun(\calC,\calO)}) \to \mathbbm{1}_{\Fun(\calC',\calO)}
$$
are equivalences where $\otimes^{\Fun(\calC,\calO)}$ and  $\otimes^{\Fun(\calC',\calO)}$ are the tensor bifunctors in the Day convolution structures and
$\mathbbm{1}_{\Fun(\calC,\calO)}, \mathbbm{1}_{\Fun(\calC',\calO)}$ are the tensor units (the transformations in question are easily described in terms of their adjoints and the lax symmetric monoidal structure of $q^*$). 
Let $F,G: \calC \to \calO$ be functors. Then  $F \otimes^{\Fun(\calC,\calO)} G$ is according to Lemma \ref{lemalifts} given by the left Kan extension 
of $F \otimes^\calO G: \calC \times \calC \to \calO$ along $\otimes^{\calC}: \calC \times \calC \to \calC$. Thus  $q_!(F \otimes^{\Fun(\calC,\calO)}G)$ as the further left Kan extension along $\calC \to \calC'$ can be obtained as the left Kan extension along the composition $q \circ \otimes^\calC$. But this is equivalent to the left Kan extension along $\otimes^{\calC'} \circ (q \times q)$ by the assumption that $q$ is symmetric monoidal. Then finally the claim follows from the fact that the left Kan extension of $F \otimes^\calO G: \calC \times \calC \to \calO$ along $q \times q$ is equivalent to $q_!F \otimes^\calO q_!G$ which follows from the assumption on the tensor product of $\calO$. The case of the unit works similar.
\end{proof}
%
%
%Then $q_!(F)$ and $q_!(G)$ are the left Kan extensions of $F$ and $G$ along $q_{\langle 1 \rangle}: \calC \to \calC'$. Then $q_!F \otimes^{\Fun(\calC',\calO)}  q_!G$ is according to Lemma \ref{lemalifts} the left Kan extension of 
%$$(q_!F \otimes^\calO q_!G): \calC' \times \calC' \to \calO $$ 
%along $\otimes^{\calC'}: \calC' \times \calC' \to \calC'$.
%We will write that as $\otimes^\calC_! ( q_!F \otimes^\calO q_!G)$
%
%in the diagram
%$$
%\xymatrix{
%\calC' \times \calC' \ar[r] \ar[d] & \calO \times \calO \ar[d] \\
%\calC' \ar@{-->}[r] & \calO
%}
%$$ 
%and the question is whether this is the same as the left Kan extension 
%Composing these two together and using that the 

%Since $q_!$ as well as the tensor bifunctors commute with colimits (in each variable) we can reduce this to checking it for functors of the form
%$$
%F_{c,x} := \Map_\calC(c,-) \otimes x : \calC \to \calO
%$$
%for $c \in \calC$ and $x \in \calO$ as these generate the functor category under colimits. But we get (by definition of $q_!$) that 
%$q_!(F_{c,x}) \simeq F_{qc,x}$. Since and thus that we have to check that the natural transformation
%$$
%q_!(F_{c,x} \otimes F_{c'.x'}) \to (F_{qc,x}) \otimes (F_{qc'.x'}) \simeq F   
%$$
%for every $c,c' \in \calC$ and $x \in \calO$
%and
%$$
%q_!(F_{\mathbbm{1_\calC},\pt}) \to F_{\mathbbm{1}_{\calC'},\pt}
%$$

Note that for a symmetric monoidal $\infty$-category $\calC$ the opposite category $\calC^{op}$ also carries a canonical structure of a symmetric monoidal $\infty$-category. This can be constructed by first straightening $\calC^\otimes \to \N\Fin$ 
to a functor $\N\Fin \to \Cat$, then post composing with $(-)^{op}: \Cat \to \Cat$ and then straightening to a symmetric monoidal $\infty$-category $(\calC^{op})^\otimes \to \N\Fin$ again. One can also give a direct simplicial model, see \cite{BarDualizing}, see also \cite[Appendix A.3]{Knudsen} for a discussion of the functoriality of this construction. 
Thus the functor category $\Fun(\calC^{op}, \calS)$ admits an extension to a presentably symmetric monoidal $\infty$-category $\Fun(\calC^{op},\calS)^\otimes$.  

\begin{proposition}[Glasman {\cite[Section 3]{Glas}}]\label{glassymyoneda}
Let $\calC^\otimes$ be a small symmetric monoidal $\infty$-category. Then the Yoneda embedding $j: \calC \to \Fun(\calC^{op}, \calS)$ admits a canonical  refinement to a symmetric monoidal functor $\calC^\otimes \to \Fun(\calC^{op}, \calS)^\otimes$. 
\end{proposition}

We can now combine this observation with the equivalence $\CAlg(\Fun(\calC^{op}, \calS)) \simeq \Fun_\lax(\calC^{op}, \calS)$ as stated in Proposition \ref{propopeee} to get

\begin{corollary} \label{refyoneda}
For $\calC$ a symmetric monoidal $\infty$-category we have a fully faithful functor 
$j': \CAlg(\calC) \to \Fun_{\lax}(\calC^{op}, \calS)$ that (up to equivalence) refines the ordinary Yoneda embedding $j$.
\end{corollary}
\begin{proof}
The only thing that we have to show is fully faithfulness. But this follows since the map $\calC^\otimes \to \Fun(\calC^{op}, \calS)^\otimes$ is fully faithful in the operadic since (i.e. as a functor of total $\infty$-categories) since the underlying functor is fully faithful by the Yoneda lemma and it is symmetric monoidal. For such operad maps all the induced functors on algebra $\infty$-categories are also fully faithful.  
\end{proof}

Note that every presentably symmetric monoidal $\infty$-category $\calD^\otimes$ admits internal mapping objects $\map_{\calD}(a,b) \in \calD$ for $a,b \in \calD$. In particular if $\calD^\otimes$ is presentably symmetric monoidal 
then according to Proposition \ref{propopeee} so is $\Fun(\calC,\calD)^\otimes$ and thus admits internal mapping object. For later use we want to give a concrete description of these internal mapping objects in $\Fun(\calC,\calD)$. It turns out to be most 
conveniently expressed using the language of ends in $\infty$-categories as discussed in \cite{fibrations} and \cite{Glas}. Let us sketch the relevant facts about ends:
\begin{enumerate}
\item
Given a functor $F: \calC^{op} \times \calC \to \calD$ the end $\int_{d \in \calC} F$ also written as $\int_{c \in \calC} F(c,c)$ is the limit over the induced functor 
$\mathrm{TwArr}(\calC) \to \calC^{op} \times \calC \xto{F} D$ where $\mathrm{TwArr}(\calC)$ is the twisted arrow $\infty$-category of $\calC$ as defined in \cite[Section 4.2]{DAG10}.
\item \label{limit} Since the end is defined as a limit is also has the properties of limits, i.e. it preserves other limits and also is preserved by all functors that preserve limits. 
\item 
If $F$ is a functor
$F: \calC^{op} \times \calC \times \calC' \to \calD$ then we can take for every $t \in \calC'$ the end and obtain a functor in $t$ which we write as $t \mapsto \int_{c \in \calC} F(c,c,t)$. To see this we 
consider $F$ as a functor $\calC^{op} \times \calC \to \Fun(\calC', \calD)$ and take the end in the functor category.
\item \label{Fubini} If we have a functor $F: \calC^{op} \times \calC \times \calC'^{op} \times \calC' \to \calD$ then we have the `Fubini' result which is the equivalence
$$
\int_{c\in \calC} \int_{c' \in \calC'} F(c,c,c',c') \simeq \int_{(c, c') \in \calC \times \calC'} F(c,c,c',c') .
$$
This follows from the equivalence $\mathrm{TwArr}(\calC \times \calC') \simeq \mathrm{TwArr}(\calC) \times \mathrm{TwArr}(\calC')$.
\item\label{mappingend}
 The mapping space in a functor category between functors $F,G: \calC \to \calD$ admits a description as an end in the $\infty$-category of spaces as follows
\begin{equation*}
\Map_{\Fun(\calC, \calD)}(F,G) \simeq \int_{c \in \calC} \Map_{\calD}(F(c),G(c))
\end{equation*}
This follows from Proposition 5.1 in \cite{fibrations}
\end{enumerate}

\begin{proposition}\label{propinternal}
Assume that $\calD^\otimes$ is a presentably symmetric monoidal $\infty$-category. Then the internal mapping object in the symmetric monoidal $\infty$-category $\Fun(\calC,\calD)^\otimes$
is given by the formula
$$
\imap_{\Fun(\calC,\calD)}(F,G)(-) \simeq \int_{d \in \calC} \imap_{\calD}\Big(F(d),G(- \otimes d)\Big)
$$
\end{proposition}
\begin{proof}
We want to verify that the functor
$$
c \mapsto \int_{d \in \calC} \imap_{\calD}\Big(F(d),G(c \otimes d)\Big)
$$ 
satisfies the universal property of the internal mapping object. To that end let us compute maps from a further functor $E: \calC \to \calD$ into it:
\begin{align*}
\Map_{\Fun(\calC, \calD)}( E, &\int_{d \in \calC} \imap_{\calD}\Big(F(d),G(- \otimes d)\Big)) &&\\
&\simeq \int_{c \in \calC} \Map_{\calD}\left(E(c), \int_{d \in \calC} \imap_{\calD}\Big(F(d),G(c \otimes d)\Big)\right)  && \text{by }\eqref{mappingend}\\
&\simeq \int_{c \in \calC} \int_{d \in \calC} \Map_\calD\left(E(c), \imap_{\calD}\Big(F(d),G(c \otimes d)\Big)\right) && \text{by } \eqref{limit} \\
& \simeq \int_{(c,d) \in \calC \times \calC} \Map_\calD\big(E(c) \otimes F(d), G(c \otimes d)\big) && \text{by } \eqref{Fubini} \\
&\simeq \Map_{\Fun(\calC \times \calC, \calD )}( \otimes^{\calD} \circ (E \times F), G \circ \otimes^\calC)  && \text{by }\eqref{mappingend} \\
&\simeq \Map_{\Fun(\calC, \calD)}(E \otimes F, G)   && \text{by Cor.}\ref{univdaybla} 
\end{align*}
The chain of equivalences is natural in $E$ which establishes the desired universal property.
\end{proof}

%%%%%%%%%%%%%%%%%%%%%%%%%%%%%%%%%%%%%%%%%%%%%%%%%%%%%%%%%%%%%%%%%%%%%%%%%%%%%%%%%%%%
\section{Stable \texorpdfstring{$\infty$}{infinity}-Operads}
%%%%%%%%%%%%%%%%%%%%%%%%%%%%%%%%%%%%%%%%%%%%%%%%%%%%%%%%%%%%%%%%%%%%%%%%%%%%%%%%%%%%

Recall that an $\infty$-category $\calC$ is called stable if it is pointed (i.e. it admits an object which is initial and terminal) and a square $\Delta^1 \times \Delta^1 \to \calC$ is a pushout precisely if it is a pullback.
The notion of stability is studied in detail in \cite[Section 1]{HA}. We now want to discuss the operadic analogue:

\begin{definition}\label{defstable}
A symmetric monoidal $\infty$-category $\calC^\otimes$ is called stably symmetric monoidal if the underlying $\infty$-category is stable and the tensor bifunctor
$\otimes: \calC \times \calC \to\calC$ preserves finite colimits separately in each variable (i.e. is exact in each variable). 
A (small) $\infty$-operad $\calO^\otimes$ is stable
if it admits a fully faithful inclusion $i: \calO^\otimes \to \calC^\otimes$ where $\calC^\otimes$ is a (small) stably symmetric monoidal $\infty$-category and such
that the essential image $i(\calO) \subseteq \calC$ is closed under finite limits.  
We let  $\Op^{\St}$ be the full subcategory of $\Op^\Lex$ (see Definition \ref{deflimits}) consisting of the stable $\infty$-operads.
 \end{definition}
 
\begin{example}\label{examples}
The category of spectra with the smash product is stably symmetric monoidal. All categories of modules over a commutative ring spectrum (such as the $\infty$-category of chain complexes) are stably symmetric monoidal. Since this example is large the convention of Remark \ref{remarksize} is applied, i.e. we work one universe higher. 

An example of a stable $\infty$-operad is formed by the suboperad $\Sp^\otimes_{\mathrm{fg}} \subseteq \Sp^\otimes$ of the $\infty$-operad of all spectra spanned by the spectra with finitely generated homotopy groups in every degree. This subcategory if clearly closed under finite limits, but not under smash products (e.g. $H\mathbb{Z} \otimes \mathrm{KU} \simeq \mathrm{KU}_\mathbb{Q}$ is not in $ \Sp_{\mathrm{fg}}$).
\end{example}

\begin{remark}\label{remarkde}
\begin{itemize}
 \item 
Note that for a stable $\infty$-operad $\calO^\otimes \to \N\Fin$ the total $\infty$-category $\calO^\otimes$ is in general not stable as an $\infty$-category. Conversely if the $\infty$-category $\calO^\otimes$ is stable then $\calO^\otimes$ is 
not necessarily stable as an $\infty$-operad. 
\item
One can give a more intrinsic characterization of stability for an operad using the theory of operadic colimits discussed in \cite[Section 3.1.1]{HA}: an $\infty$-operad $\calO^\otimes$ is stable if and only if it 
admits finite operadic colimits for diagrams in $\calO$ and finite operadic limits and its underlying $\infty$-category is stable. We will not need this characterization and therefore avoid the theory of operadic colimits. 
The key to see the equivalence to our definition is to give an 
independent proof of Proposition \ref{proppointad} below using this alternative characterization.  
\item
The underlying $\infty$-category $\calO$ of a stable $\infty$-operad is stable. This follows since $i(\calO) \subseteq \calC$ is a stable subcategory. But in general it is not sufficient that $\calO$ is stable as an $\infty$-category for $\calO^\otimes$ to be stable as an $\infty$-operad. 
However under the assumption that $\calO^\otimes$ admits all 
operadic colimits and limits it is equivalent as noted before. Again we do not prove this here.
\item A symmetric monoidal $\infty$-category $\calC^\otimes$ is stably symmetric monoidal precisely if it is stable as an $\infty$-operad. 
One direction is clear taking the identity as an embedding. For the other direction
assume that $i: \calC^\otimes \to \calC'^\otimes$ is an  fully faithful embedding as in the definition. In particular the underlying functor preserves all finite colimits and limits. 
Then for every finite colimit in $\calC$, which we will write abusively as $\colim c_i$, and objects $c, d \in \calC$ we have the following chain of equivalences
\begin{align*}
\Map_{\calC}(c \otimes_\calC \colim c_i, d) &\simeq \Mul_{\calC}(c,\colim c_i;d) \\
&\simeq \Mul_{\calC'}(i(c),i(\colim c_i);i(d)) \\
&\simeq \Map_{\calC'}(i(c) \otimes_{\calC'} i(\colim c_i);i(d)) \\
&\simeq \Map_{\calC'}(\colim (i(c) \otimes_{\calC'} i(c_i));i(d)) \\
&\simeq \lim \Map_{\calC'}(i(c) \otimes_{\calC'} i(c_i);i(d)) \\
&\simeq \lim \Mul_{\calC'}(i(c),i(c_i);i(d)) \\
&\simeq \lim \Mul_{\calC}(c,c_i;d) \\
&\simeq \lim \Map_{\calC}(c \otimes_{\calC} c_i;d)
\end{align*}
This shows that the canonical cone exhibits $c \otimes \colim c_i$ as the colimit of $c \otimes c_i$ in $\calC$. Therefore the tensor product of $\calC$ preserves colimits separately in the second variable and thus by symmetry in both.
\item 
There is also a definition of stability for $\infty$-operads in \cite[Section 6.2.4]{HA}. This notion differs from ours, in that Lurie calls an $\infty$-operad $\calO^\otimes$ stable if the functor $\calO^\otimes \to \N\Fin$ is locally Cartesian, 
the underlying $\infty$-category
is stable and the associated tensor functor $\otimes^n: \calO^n \to \calO$ (which exists by the  locally Cartesian assumption) is exact in each variable separately. One can show that every such $\infty$-operad is stable in our sense. 
but the converse is not true, i.e. our notion of stability is more general. 
\end{itemize}
\end{remark}

\begin{lemma}\label{pointed}
A small $\infty$-operad $\calO$  is stable if and only if  
it admits a fully faithful inclusion $i: \calO^\otimes \to \calC^\otimes$ where $\calC^\otimes$ is presentably symmetric monoidal, $\calC$ is stable and the essential image $i(\calO) \subseteq \calC$ is closed under finite limits.
\end{lemma}
\begin{proof}
Assume $\calO$ is stable. By definition of stability we find an operadically fully faithful embedding $i: \calO^\otimes \to \calC^\otimes$ where $\calC^\otimes$ is stably symmetric monoidal. Now we form the Ind-completion 
$\Ind(\calC)$. According to \cite[Proposition 1.1.3.6]{HTT} and \cite[Corollary 4.8.1.13]{HA} it is stable, presentably symmetric monoidal  such that the inclusion $\calC \to \Ind(\calC)$ admits a symmetric monoidal refinement. 
Then $\calO^\otimes \to \calC^\otimes \to \Ind(\calC)^\otimes$ is the desired embedding.

For the converse assume that we have an embedding $\calO^\otimes \to \calC^\otimes$ as described. Then the underlying embedding $\calO \to \calC$ factors through the $\kappa$-compact objects $\calC^\kappa$
for some cardinal $\kappa$ since $\calO$ is small. Enlarging $\kappa$ if necessary we can assume that $\calC^\kappa$ is closed
under tensor products and finite limits. Since $\calC^\kappa$ is closed under colimits it is also stable. Thus $(\calC^\kappa)^\otimes$ can be used to see that $\calO^\otimes$ is stable.
\end{proof}

Now we want to explain how to `stabilize' an $\infty$-operad. Therefore let us first recall the stabilization of $\infty$-categories.

Denote by $\psf$ the $\infty$-category of finite pointed CW complexes. These are by definition precisely the spaces which can be obtained under finite (homotopy) colimits from the point. 
This category admits all finite colimits, but it is not idempotent complete as witnessed by Wall's finiteness obstruction. Let $\calC$ be an $\infty$-category which admits finite limits. 
Then the category of spectrum objects $\Sp(\calC)$ is defined as the full subcategory of
$\Fun(\psf, \calC)$ spanned by the reduced excisive functors. These are the functors $F: \psf \to \calC$ such that $F(*)$ is terminal in $\calC$ and such that $F$ carries pushout squares in $\psf$ to 
pullback squares in $\calC$.  For example we can obtain the $\infty$-category $\Sp$ of spectra as $\Sp(\calS)$. There is a canonical forgetful functor 
$\Omega^\infty: \Sp(\calC) \to \calC$ obtained by evaluation at $S^0$. See \cite[Section 1.4.2]{HA} for a more detailed discussion.

The $\infty$-category $\psf$ admits a symmetric monoidal structure  given by smash product. We denote the resulting $\infty$-operad by $(\psf)^\otimes \to \N\Fin$.  
\begin{construction}
Let $\calO^\otimes$ be an $\infty$-operad that admits finite operadic limits. We denote by $\hom_{{/\N\Fin}}\left((\psf)^\otimes, \calO^\otimes\right)$ the internal hom in simplicial sets over $\N\Fin$ (cf. first paragraph of Section \ref{day}).
In particular the fibre of  $\hom_{{/\N\Fin}}\left((\psf)^\otimes, \calO^\otimes\right)$ over $\langle n \rangle$ is given by the $\infty$-category
$$\Fun\left((\psf)^\otimes_{\langle n \rangle},\calO^\otimes_{\langle n \rangle}\right)$$
We define the simplicial set $\SP(\calO)^\otimes$ over $\N\Fin$ to be the full simplicial subset whose fibre over $\langle n \rangle$ is spanned by those functors
$$
F: (\psf)^\otimes_{\langle n \rangle} \to \calO^\otimes_{\langle n \rangle}
$$
such that under the equivalences 
$$
\rho: (\psf)^\otimes_{\langle n \rangle} \xto{\sim} \prod_n \Fin \qquad \rho': \calO^\otimes_{\langle n \rangle} \xto{\sim} \prod_n \calO
$$
 we have an equivalence of functors $\rho' \circ F \simeq (F_1 \times...\times F_n) \circ \rho$
 where all $F_i: \psf \to \calO$ lie in $\SP(\calO) \subseteq \Fun(\psf,\calO)$ i.e. are reduced and excisive. 
 \end{construction}

It is clear that by construction the underlying $\infty$-category of $\Sp(\calO)^\otimes$ is given by $\Sp(\calO)$. This explains our notation. 

\begin{proposition}\label{pointer_operad}
For every $\infty$-operad $\calO^\otimes$ which admits finite operadic limits the simplicial set $\Sp(\calO)^\otimes \to \N \Fin$ is also an $\infty$-operad which admits finite operadic limits.
For every map $F: \calO^\otimes \to \calO'^\otimes$ in $\Op^\Lex$ the evident map $F_*: \Sp(\calO)^\otimes \to \Sp(\calO)^\otimes$ induced by post-composition with $F$ is also in $\Op^{\Lex}$.
\end{proposition}
\begin{proof}
By definition the simplicial set $\Sp(\calO)^\otimes$ is a full simplicial subset of the Day convolution $\Fun(\psf, \calO)^\otimes \to \N\Fin$ as described in Definition \ref{definitionday}. 
As a subset it is precisely spanned by the products of reduced excisive functors. In general if $\calP^\otimes$ is an $\infty$-operad and $\calP_0 \subseteq \calP$ is a full simplicial subset then the corresponding full simplicial subset $\calP^\otimes_0 \subseteq \calP^\otimes$ spanned by products of $0$-simplices in $\calP_0$ is again an $\infty$-operad as one easily shows using the definition of $\infty$-operads.
Therefore it follows immediately from the fact that $\Fun(\psf, \calO)^\otimes$ is an $\infty$-operad (as shown in Proposition \ref{propopeee}) that $\Sp(\calO)^\otimes$ is again an $\infty$-operad.

With the same argument we get that the induced morphism $F_*: \Sp(\calO)^\otimes \to \Sp(\calO')^\otimes$ is a morphism of $\infty$-operad from the fact that the functor $\Fun(\psf,\calO)^\otimes \to \Fun(\psf, \calO')^\otimes$
is a map of $\infty$-operads as shown in Corollary \ref{morphism_co}. 

It only remains to see that $\Sp(\calO)^\otimes$ admits finite limits and that the morphism $F_*$ preserves finite limits. For that purpose we embed $i: \calO^\otimes \to \calC^\otimes$ into 
a presentably symmetric monoidal $\infty$-category such that the image is closed under finite limits using Proposition \ref{limits}. Then clearly the induced morphism $i_*: \Sp(\calO)^\otimes \to \Sp(\calC)^\otimes$ is also fully faithful and the essential image
$i_*(\Sp(\calO)) \subseteq \Sp(\calC)$ is closed under finite limits (as the forgetful functor $\Sp(\calO) \to \calO$ preserves all limits). Again from Proposition \ref{limits} we get that $\Sp(\calO)^\otimes$ has finite operadic limits. 
The fact that the  morphism $F_*$ preserves finite operadic limits follows since the underlying morphism $(F_*) _{\langle 1 \rangle}: \Sp(\calO) \to \Sp(\calO')$ of $\infty$-categories preserves limits and since both $\infty$-operads have all finite operadic limits.% which is stated in \cite[suchen]{HA}.
\end{proof}

\begin{construction}\label{constrforget}
There is a morphism $\Omega^\infty: \Sp(\calO)^\otimes \to \calO^\otimes$ informally given by evaluation at $S^0$. Concretely we define it as the composition
 $$
 \Sp(\calO)^\otimes \subseteq \hom_{{/\N\Fin}}\left((\psf)^\otimes, \calO^\otimes\right) \xrightarrow{s^*} \hom_{{/\N\Fin}}\left(\N\Fin, \calO^\otimes\right) \cong \calO^\otimes
 $$
 where $s: \N\Fin \to (\psf)^\otimes$ is the section of $(\psf)^\otimes \to \N \Fin$ given by the commutative algebra object $S^0$ in $\psf$. The fact that it is a map of $\infty$-operads 
follows immediately from Corollary \ref{morphism_contra} since it is the restriction of the morphism $\Fun(\psf,\calO)^\otimes \to \Fun(\mathbb{E}_\infty,\calO)^\otimes \cong \calO^\otimes$. Here $\mathbb{E}_\infty^\otimes$ is the $\infty$-operad given by the identity $\N\Fin \to \N\Fin$.
The functor $\Omega^\infty$ also preserves finite operadic limits
since the underlying functor of $\infty$-categories does.
\end{construction}

\begin{corollary}\label{corheute}
The assignment $\calO^\otimes \mapsto \Sp(\calO)^\otimes$ refines to an endofunctor $\Sp: \Op^\Lex \to \Op^\Lex$ with a natural transformation $\Omega^\infty: \Sp \to \id$.
\end{corollary}
\begin{proof}
By Proposition \ref{pointer_operad} we see that the assignment $\Sp(-)$ determines a functor from the 1-category of $\infty$-operads with finite operadic limits and limit preserving $\infty$-operad maps to itself.
This 1-category admits a simplicial enrichment given for operads $\calO^\otimes$ and $\calO^\otimes$ by the maximal Kan complex contained inside the $\infty$-category $\Alg^\Lex_\calO(\calO')$. 
Then by construction the functor $\Sp(-)$ evidently induces a simplicial functor. But the $\infty$-category $\Op^\Lex$ is the homotopy coherent nerve of that simplicial category. Therefore we get a functor of the associated $\infty$-categories. 
The same arguing shows that the transformation $\Omega^\infty$ which comes from a simplicial transformation exists.
\end{proof}

\begin{proposition}\label{propsymm}
Assume that $\calC^\otimes$ is presentably symmetric monoidal. Then so is $\Sp(\calC)^\otimes$ and the forgetful map $\Omega^\infty: \Sp(\calC)^\otimes \to \calC^\otimes$ 
admits an operadic left adjoint $\Sigma^\infty_+: \calC^\otimes \to \Sp(\calC)^\otimes$ which is symmetric monoidal.
\end{proposition}
\begin{proof}
By Proposition \ref{propopeee}, the operad $\Fun(\psf,\calC)^\otimes$ is presentably symmetric monoidal and by Corollary \ref{morphism_contra} the map 

$$
s^*: \Fun(\psf,\calC)^\otimes \to \calC^\otimes
$$ 
admits an operadic left adjoint which is symmetric monoidal. We have to show that the composite 
$$
\Omega^\infty: \Sp(\calC)^\otimes \subseteq \Fun(\psf,\calC)^\otimes \xto{s^*} \calC^\otimes
$$ 
admits an operadic left adjoint which is symmetric monoidal. Therefore it suffices to show that the inclusion $\Sp(\calC)^\otimes \to \Fun(\psf,\calC)^\otimes$ does.
The underlying subcategory $\Sp(\calC)_* \subseteq \Fun(\psf,\calC)$ is a 
localizing subcategory. Thus we can use
the criterion established in Example \ref{exampleeins} \eqref{pointudrei}. 
Therefore Let $F,G: \psf \to \calC$ be functors and assume that $G$ is reduced and excisive. 
We have to show that the internal mapping object $\map_{\Fun(\psf,\calC)}(F,G)$ is also reduced and excisive. But according to  
Proposition \ref{propinternal} we have 
\begin{equation}\label{formhom}
\map_{\Fun(\psf,\calC)}(F,G)(c) \simeq  \int_{d \in \psf} \map_\calC(F(d), G(c \wedge d))
\end{equation}
First if $c \in \psf$ is terminal (i.e. a contractible space) then $c \wedge d$ is also terminal and therefore also $G(c \wedge d)$ and $\map_\calC(F(d), G(c \wedge d))$. As ends preserve limits we get from \eqref{formhom} that the functor $\map_{\Fun(\psf,\calC)}(F,G)$ is reduced.  

If we have a pushout square $p: \Delta^1 \times \Delta^1 \to \psf$ then $p \otimes d$ is also a pushout for every $d \in \psf$. Thus $G(p \otimes d)$ is a pullback and hence also $\map_{\calC}(F(d),G(p \otimes d))$. We again use \eqref{formhom} and the fact that ends preserve
pullbacks to deduce that $\map_{\Fun(\psf,\calC)}(F,G)(p)$ is a pullback. This shows finally that $\map_{\Fun(\psf,\calC)}(F,G)$ is excisive and finishes the proof. 
\end{proof}

\begin{corollary}\label{cormor}
For $\calC^\otimes \simeq \calS^\times$ the $\infty$-category of spaces with cartesian product we have an equivalence of $\infty$-operads $\Sp(\calC)^\otimes \simeq \Sp^\otimes$ where $\Sp^\otimes$ is the symmetric monoidal 
$\infty$-category of spectra with smash product.
\end{corollary}
\begin{proof}
  By Proposition \ref{propsymm} we know that the left adjoint $\Sigma^\infty_+: \calS \to \Sp$ admits a symmetric monoidal structure. Then the equivalence with the smash product of spectra follows from the uniqueness of the symmetric monoidal
structure as shown in \cite[Corollary 4.8.2.19]{HA}, see also \cite[Theorem 5.1]{GGN}.
\end{proof}

\begin{thm}\label{proppointad}
For every $\infty$-operad $\calO^\otimes$ which admits finite operadic limits the $\infty$-operad $\Sp(\calO)^\otimes$ is stable.
If $\calO^\otimes$ is stable then the functor $\Omega^\infty: \Sp(\calO)^\otimes \to \calO^\otimes$ is an equivalence of $\infty$-operads.
\end{thm}
\begin{proof}
To show that $\Sp(\calO)^\otimes$ is stable we first assume that $\calO^\otimes \simeq \calC^\otimes$ is presentably symmetric monoidal (and consequently the underlying $\infty$-category $\calC$ admits all limits). 
Then according to Proposition \ref{propsymm} the $\infty$-operad
$\Sp(\calC)^\otimes$ is presentably symmetric monoidal as well. The underlying $\infty$-category $\Sp(\calC)$ is stable as shown in \cite[Corollary 1.4.2.17]{HA}. Thus we conclude that $\Sp(\calC)^\otimes$ 
is stably symmetric monoidal and therefore also stable as an $\infty$-operad.

For general $\calO^\otimes$ we find  a finite limit preserving, fully faithful embedding $\calO^\otimes \to \calC^\otimes$ into a presentably symmetric monoidal $\infty$-category $\calC^\otimes$ 
using Proposition \ref{limits}.
Then $\Sp(\calO)^\otimes\to \Sp(\calC)^\otimes$ is also fully faithful which follows immediately from the construction.  We conclude using Lemma \ref{pointed} that $\Sp(\calO)^\otimes$ is stable as well.
This shows the first part of the proposition. 

We now want to show that for a stable $\infty$-operad $\calO^\otimes$ the morphism $\Omega^\infty: \Sp(\calO)^\otimes \to \calO^\otimes$ is an equivalence of $\infty$-operads. 
We assume first that $\calO^\otimes \simeq \calC^\otimes$ is presentably symmetric monoidal. Then according to \ref{propsymm} we find that 
$$
\Omega^\infty: \Sp(\calC)^\otimes \to \calC^\otimes
$$ 
admits an operadic left adjoint $\Sigma^\infty_+: \calC^\otimes \to \Sp(\calC)^\otimes$ which is symmetric monoidal. 
In \cite[Proposition 1.4.2.21]{HA} it is shown that
the underlying functor $\Omega^\infty:\Sp(\calC) \to \calC$ and  $\Sigma^\infty_+: \calC \to \Sp(\calC)$ are inverse equivalences (since $\calC$ is stable).  But a symmetric monoidal functor between symmetric monoidal $\infty$-categories is an 
equivalence of $\infty$-operads precisely if 
the underlying functor of $\infty$-categories is. Thus $\Sigma^\infty_+: \calC^\otimes \to \Sp(\calC)^\otimes$ is an equivalence of $\infty$-operads. Therefore also $\Omega^\infty$ (as the right adjoint) is an equivalence of $\infty$-operads. 

Finally for a general stable $\infty$-operad $\calO$ choose an embedding into a presentably symmetric monoidal $\infty$-operad $\calC^\otimes$ according to Lemma \ref{pointed}. We get a commutative square
$$
\xymatrix{
\Sp(\calO^\otimes) \ar[r] \ar[d]^{\Omega^\infty} & \Sp(\calC)^\otimes\ar[d]^{{\Omega}^\infty} \\
\calO^\otimes \ar[r] & \calC^\otimes
}$$
in which the right hand morphism ${\Omega}^\infty$ is an equivalence of $\infty$-operads and the underlying morphism of the left hand morphism  $\Omega^\infty$ is an equivalence of $\infty$-categories. 
Since the horizontal morphisms are both fully faithful this implies that the left hand morphism $\Omega^\infty$ is essentially surjective and 
operadically fully faithful. But this implies that it is an  equivalence of $\infty$-operads.
\end{proof}

\begin{corollary}
The endofunctor $\Sp: \Op^\Lex \to \Op^\Lex$ together with the transformation $\Omega^\infty: \Sp \to \id$ is a colocalization onto the full subcategory of stable $\infty$-operads. 
\end{corollary}
\begin{proof}
The follows immediately from the last Theorem \ref{proppointad}.
\end{proof}

\begin{corollary}\label{corimo}
For every stable $\infty$-operad $\calO^\otimes$ the functor
$$
\Alg^\Lex_{\calO}\left(\Sp\right) \xrightarrow{\Omega^\infty} \Alg^\Lex_{\calO}\left(\calS\right)
$$
induced by the forgetful map $\Omega^\infty: \Sp^\otimes \to \calS^\times$ is an equivalence of $\infty$-categories.
\end{corollary}
\begin{proof}
Using the functoriality statement in Proposition \ref{pointer_operad} we get a functor
$$
\Alg^\Lex_{\calO}\left(\calS\right) \to \Alg^\Lex_{\Sp(\calO)}\left(\Sp\right)
$$
Since $\calO^\otimes$ is stable the target of that functor is equivalent to $\Alg^\Lex_{\calO}\left(\Sp\right)$ and it is clear that the functor is an inverse to the one in question.
\end{proof}

\begin{remark}\label{Lurie}
Our stablilization construction $\calO^\otimes \mapsto \Sp(\calO)^\otimes$ is on the level of simplicial sets isomorphic to the construction given in \cite[Construction 6.2.5.20]{HA}. 
Lurie shows under a stronger assumption than we have (namely differentiablity \cite[Definition 6.2.4.11]{HA}), that this gives a stabilization for his notion of stable $\infty$-operads which is special case of ours. He also proves similar results to ours: 
for example \cite[Proposition 6.2.4.15]{HA} is the direct analogue of Corollary \ref{corimo} and \cite[Example 6.2.4.17]{HA} is the direct analogue of Corollary \ref{cormor}. 
In particular we see a posteriori that for an $\infty$-operad which is differentiable the $\infty$-operad $\Sp(\calO)^\otimes$ is stable in the sense of Lurie. 

We will give in the next section a treatment for the case of (pre)additive and pointed operads that is completely parallel to the treatment in this section. We are not aware of a treatment of these cases along the lines of Lurie's strategy.
\end{remark}

%%%%%%%%%%%%%%%%%%%%%%%%%%%%%%%%%%%%%%%%%%%%%%%%%%%%%%%%%%%%%%%%%%%%%%%%%%%%%%%%%%%%%
\section{Pointed and (Pre)additive  \texorpdfstring{$\infty$}{infinity}-Operads}\label{secbla}
%%%%%%%%%%%%%%%%%%%%%%%%%%%%%%%%%%%%%%%%%%%%%%%%%%%%%%%%%%%%%%%%%%%%%%%%%%%%%%%%%%%%%

In this section we want to establish a picture similar to the one discussed for stable $\infty$-operads, but for (pre)additive and pointed $\infty$-operads. 
Since most of the proofs and result are analogous we try to only mention the crucial differences and not repeat everything.

Recall that an $\infty$-category $\calC$ is called pointed if 
it admits a zero object, that is an object $0 \in \calC$ which is initial and terminal. A pointed $\infty$-category $\calC$ is called preadditive if 
for every pair of objects $a,b \in \calC$ the canonical morphism 
$$
\begin{psmallmatrix} 
\id & 0 \\ 0 & \id 
\end{psmallmatrix} :~
a \sqcup b \to a \times b
$$
is an equivalence. In this case we write $a \oplus b$ for the biproduct. 
A preadditive $\infty$-category is called additive if for every object $a \in \calC$ the shear map 
$$
\begin{psmallmatrix} 
\id & \id \\ 0 & \id 
\end{psmallmatrix} :~
a \oplus a \to a \oplus a
$$ 
is an equivalence. For example every stable $\infty$-category is additive and in particular also preaddtive and pointed. See \cite[Section 2]{GGN} for a more detailed discussion. 

\begin{definition}
A symmetric monoidal $\infty$-category $\calC^\otimes$ is called pointed if the underlying $\infty$-category is pointed and the tensor bifunctor
$\otimes: \calC \times \calC \to\calC$ has the property that $c \otimes 0 \simeq 0 \simeq 0 \otimes c$ for every $c \in \calC$. It is called (pre)additive if the underlying $\infty$-category is (pre)additive
and the tensor bifunctor preserves coproducts separately in each variable.

An $\infty$-operad $\calO^\otimes$ is called pointed 
if it admits a fully faithful inclusion $i: \calO^\otimes \to \calC^\otimes$ where $\calC^\otimes$ is a pointed symmetric monoidal $\infty$-category and the
essential image $i(\calO) \subseteq \calC$ contains the zero object. It is called (pre)additive if $\calC^\otimes$ can be chosen to be (pre)additive and such that the essential
image $i(\calO) \subseteq \calC$ is closed under finite products.  

We define $\Op^\pt$ to be the full subcategory of $\Op^*$ consisting of the pointed $\infty$-operads and $\Op^{\Pre}$ and $\Op^{\Add}$ to be be the full subcategories of $\Op^\Pi$ consisting of the preadditive and additive $\infty$-operads (see Definition \ref{deflimits} for the Definition of $\Op^*$ and  $\Op^\Pi$). 
 \end{definition}
 
% operadic products and a product preserving, operadically fully faithful inclusion $i: \calO^\otimes \to \calC^\otimes$ 
% such that 
% \begin{itemize}
% \item The underlying $\infty$-category $\calC := \calC^ \otimes_{\langle 1 \rangle}$ is additive, i.e. $\mathrm{Ho}(\calC)$ is additive in the ordinary sense.
% \item $\calC^\otimes$ is symmetric monoidal and the tensor functor $\otimes: \calC \times \calC \to \calC$ preserves coproducts in each variable seperately.
% \end{itemize}
% Note that the underlying $\infty$-category of an additive operad is additive. We define 
% \end{definition}

\begin{example}
Since every stable $\infty$-operad is pointed and additive, the examples mentioned in Example \ref{examples} are also examples here, in particular the symmetric monoidal $\infty$-categories of spectra and chain complexes are pointed and additive. 
An example of an additive symmetric monoidal $\infty$-category is the $\infty$-category of connective spectra with the smash product. 
The ordinary category of abelian groups with tensor product is an additive symmetric monoidal $\infty$-category. The category of abelian monoids with the tensor product of abelian monoids is a preadditive but not additive symmetric monoidal
$\infty$-category. The category of pointed sets and the $\infty$-category of pointed spaces with the smash product are examples of pointed symmetric monoidal $\infty$-categories.
\end{example}

\begin{remark}
The obvious analogues of Remark \ref{remarkde} are true. For example one can characterize preadditivity as the existence of operadic biproducts and pointed as the existence of an operadic zero object.
 Also whether a symmetric monoidal $\infty$-category is (pre)additive or pointed as an operad or 
as a symmetric monoidal $\infty$-category is the same. 
Note also that being pointed for an $\infty$-operad is not same as unital, which means that the $\infty$-category $\calO^\otimes$ is pointed as discussed in \cite[Section 2.3.1]{HA}.
\end{remark}
\begin{lemma}
A small $\infty$-operad $\calO$  is pointed (resp. (pre)additive) if and only if  
it admits a fully faithful inclusion $i: \calO^\otimes \to \calC^\otimes$ where $\calC^\otimes$ is presentably symmetric monoidal, $\calC$ is pointed (resp. (pre)additive) and the essential image $i(\calO) \subseteq \calC$ contains the zero objects (resp. is closed under finite limits).
\end{lemma}
\begin{proof}
We can  choose by definition an embedding $\calO^\otimes \to \calC'^\otimes$ where $\calC'^\otimes$ is (pre) additive or pointed symmetric monoidal. Then we form the 
$\infty$-category $\mathcal{P}^{\text{all}}_\Sigma \calC'$ or $\mathcal{P}^{\text{all}}_\pt \calC'$ which is the universal category obtained from $\calC'$ by adding all colimits while keeping finite coproducts (resp. the initial object). It admits a tensor product which extends the tensor product of $\calC'$ (under Yoneda) by \cite[Proposition 4.8.1.10]{HA} and turns it into a presentably symmetric monoidal $\infty$-category. Since the inclusion preserves finite coproducts it suffices to show that the $\infty$-categories $\mathcal{P}^{\text{all}}_\Sigma \calC'$ and $\mathcal{P}^{\text{all}}_\pt \calC'$ are (pre)additive resp. pointed. To see this we use that the underlying $\infty$ categories of  $\mathcal{P}^{\text{all}}_\Sigma \calC'$ and $\mathcal{P}^{\text{all}}_\pt \calC'$
can be described explicitly as the $\infty$-categories of functors $\Fun^\Pi(\calC'^{\op},\calS)$ and $\Fun^*(\calC'^{\op},\calS)$ as shown in \cite[Section 5.3.6]{HTT}. These are in fact (pre)additive resp. pointed as shown in \cite{GGN}.
\end{proof} 

Now we want to define for every $\infty$-operad $\calO^\otimes$ with a terminal object resp. finite products new operads $\calO^\otimes_*$ of pointed objects in $\calO^\otimes$ resp. $\CMon(\calO)^\otimes$ of commutative monoids and
$\CGrp(\calO^\otimes)$ of commutative groups in $\calO^\otimes$. These are the universal pointed resp. (pre)additive $\infty$-operad obtained from $\calO^\otimes$ as we will show eventually. 

First recall  that for an $\infty$-category $\calC$ with a terminal object the category of pointed objects $\calC_*$ is defined to be the full subcategory of the arrow category $\Fun(\Delta^1,\calC)$ consisting of those arrows whose source
is a terminal object. For an $\infty$-category $\calC$ which admits finite products a commutative monoid in $\calC$ is a functor $F:  \N\Fin \to \calC$ which satisfies the Segal condition, that is for every $n \geq 0$
the canonical morphism
$$
(\rho^1_*,...,\rho^n_*): F\langle n \rangle \to F\langle 1 \rangle \times ... \times F\langle 1 \rangle
$$
is an equivalence where $\rho^i: \langle n \rangle \to \langle 1 \rangle$ is the morphism with $\rho^i(i) = 1$ and $\rho^i(k) = 0$ for $k \neq i$. A commutative monoid $F$ in $\calC$ is a commutative group if additionally the morphism
$$
(\rho^1_*, \mu_*): F\langle 2 \rangle \to F\langle 1 \rangle \times F(\langle 1 \rangle
$$
is an equivalence where $\mu: \langle 2 \rangle \to \langle 1 \rangle$ is defined as $\mu(1) = 1$ and $\mu(2) = 2$. The $\infty$-categories $\CMon(\calC)$ and $\CGrp(\calC)$ of commutative monoids and groups in $\calC$ are the full 
subcategories $\CGrp(\calC) \subseteq \CMon(\calC) \subseteq \Fun(\N\Fin, \calC)$ consisting of the group resp. monoid objects. \\
 
The $\infty$-category $\Delta^1$, which is in fact the nerve of a poset, admits a symmetric monoidal structure given by categorical product, which is explicitly given by taking the minimum in the poset. 
If we name the elements by $0$ and $1$ we find that 
$$ 0 \otimes 0 = 0 \qquad 0 \otimes 1 = 0 \qquad 1 \otimes 1 = 1$$
and this already uniquely determines the symmetric monoidal structure. We denote the corresponding symmetric monoidal $\infty$-category by $(\Delta^1)^\otimes \to N\Fin$.
The category $\Fin$ also admits a symmetric monoidal structure given by smash product. We have 
$\langle n \rangle \otimes \langle m \rangle  \simeq \langle nm \rangle. $
 We denote the resulting symmetric monoidal $\infty$-category by $\N \Fin^\otimes \to \N \Fin$.  

\begin{construction}
Let $\calO^\otimes$ be an $\infty$-operad that admits an operadically terminal object. 
We define the simplicial set $\calO_*^\otimes$ over $\N\Fin$ to be the subset of the internal hom
 $\hom_{{/\N\Fin}}\left((\Delta^1)^\otimes, \calO^\otimes\right)$
whose fibre over $\langle n \rangle$ is spanned by those functors which correspond 
under the equivalence 
$$
\hom_{{/\N\Fin}}\left((\Delta^1)^\otimes, \calO^\otimes\right)_{\langle n \rangle} \simeq \Fun\left((\Delta^1)^\otimes_{\langle n \rangle},\calO^\otimes_{\langle n \rangle}\right) 
\simeq \Fun\left(\prod_n \Delta^1, \prod_n \calO\right)
$$
to functors $F_1 \times...\times F_n$ where $F_i: \Delta^1 \to \calO$ lies in $\calO_* \subset \Fun(\Delta^1,\calO)$ i.e. $F_i(0)$ is a terminal object in $\calO$ for every $i$.

Let $\calO^\otimes$ be an $\infty$-operad that admits finite operadic products. 
We define the simplicial sets $\CMon(\calO)^\otimes$ and $\CGrp(\calO)^\otimes$ over $\N\Fin$ to be the full simplicial subset of $\hom_{{/\N\Fin}}\left(\N\Fin^\otimes, \calO^\otimes\right)$
whose fibres over $\langle n \rangle$ are spanned by those functors which correspond 
under the equivalence 
$$
\hom_{{/\N\Fin}}\left((\N\Fin)^\otimes, \calO^\otimes\right)_{\langle n \rangle} \simeq
\Fun\left((\N\Fin)^\otimes_{\langle n \rangle},\calO^\otimes_{\langle n \rangle}\right) 
\simeq \Fun\left(\prod_n \N\Fin, \prod_n \calO\right)$$ 
to functors $F_1 \times...\times F_n$ where $F_i: \N\Fin \to \calO$ lies in $\CMon(\calO)$ resp. $\CGrp(\calO)$.
\end{construction}
Now the exact same argument as in Proposition \ref{pointer_operad} shows that $\calO^\otimes_*$ is an $\infty$-operad which lies in $\Op^*$ and that
$\CMon(\calO)^\otimes$ and $\CGrp(\calO)^\otimes$ are $\infty$-operads in $\Op^\Pi$. We again as in Construction \ref{constrforget} obtain canonical forgetful maps
$$ \calO_*^\otimes \to \calO^\otimes \qquad \text{and} \qquad  \CGrp(\calO)^\otimes \subseteq \CMon(\calO)^\otimes \to \calO^\otimes$$
which are given by evaluation at the commutative algebra object $[1] \in \Delta^1$ and $\langle 1 \rangle \in \N\Fin$.
The analogue of Corollary \ref{corheute} shows that we obtain endofunctors
$\CMon, \CGrp: \Op^\Pi \to \Op^\Pi$  and $(-)_*: \Op^* \to \Op^*$ with transformations $\CGrp \to \CMon \to \id$ and $(-)_* \to \id$.

Now we want to show an analogue of Proposition \ref{propsymm}. This also works very similar to the proof given there, but since there are some changes we want to spell it out.

\begin{proposition}\label{pressymerquicdk}
Assume that $\calC^\otimes$ is a presentably symmetric monoidal $\infty$-category. Then so are $\calC^\otimes_*$, $\CGrp(\calC)^\otimes$ and $\CMon(\calC)^\otimes$. Moreover the operad morphisms 
$$ 
\CGrp(\calC)^\otimes \to \CMon(\calC)^\otimes \to  \calC^\otimes \qquad \text{and} \qquad \calC^\otimes_*  \to \calC
$$
admit operadic left adjoints which are symmetric monoidal.
\end{proposition}
\begin{proof}
By Proposition \ref{propopeee}, the operads $\Fun(\Delta^1,\calC)^\otimes$ and $\Fun(\N\Fin,\calC)^\otimes$ are presentably symmetric monoidal and by Corollary \ref{morphism_contra} the maps
$$
 \Fun(\Delta^1,\calC)^\otimes \to \calC^\otimes \qquad \text{and} \qquad  \Fun(\N\Fin,\calC)^\otimes \to \calC^\otimes
$$ 
admit operadic left adjoints which are symmetric monoidal. Thus again as in Proposition \ref{propsymm} it suffices to prove that the localizations 
$$
\Fun(\Delta^1,\calC) \to \calC_* \qquad \text{and}  \qquad \Fun(\N\Fin,\calC) \to \CMon(\calC) \to \CGrp(\calC)
$$
of the underlying $\infty$-categories are compatible with the symmetric monoidal structures. 
We want to use
the criterion established in Example \ref{exampleeins} \eqref{pointudrei}. 

Let us start with the case of pointed objects. Let $F,G: \Delta^1 \to \calC$ be functors and assume that $G$ is a pointed object, i.e. $G(0)$ is a terminal object in $\calC$. Then we use Proposition \ref{propinternal} 
to find 
\begin{equation}
\map_{\Fun(\Delta^1,\calC)}(F,G)(0) \simeq  \int_{d \in \Delta^1} \map_\calC(F(d), G(0)) 
\end{equation}
Since $G(0)$ is terminal and ends as well as mapping spaces preserve terminal objects this 
this shows that $\map_{\Fun(\Delta^1,\calC)}(F,G)$ is again a pointed object. 

Now we come to the second case. Let $F,G: \N\Fin \to \calC$ be functors and assume that $G$ is a commutative monoid (group) object. 
We have to show that the internal mapping object $\map_{\Fun(\N\Fin,\calC)}(F,G)$ is also a commutative monoid (group) object. But according to  
Proposition \ref{propinternal} we have 
\begin{equation}\label{endformular}
\map_{\Fun(\N\Fin,\calC)}(F,G)(c) \simeq  \int_{d \in \N\Fin} \map_\calC(F(d), G(c \wedge d))
\end{equation}
For $\langle n \rangle \in \N\Fin$ we have that the canonical map $G\big(\langle n \rangle \wedge d\big) \to G( d)^n$ is an equivalence for every $d$ since $G$ is a commutative monoid object (and the same for a group object and $n =2$ in case of the shear map). The mapping object and the end in \eqref{endformular} commute with products. This implies that $\map(F,G)$ is also a monoid (group) object. 
\end{proof}

From Proposition \ref{pressymerquicdk} we immediately get as in  Corollary \ref{cormor} invoking the uniqueness results from \cite{GGN} that we have equivalences
$$\CGrp(\calS)^\otimes \simeq (\Sp^{\geq 0})^{\otimes} \qquad \text{and}\qquad \calS_*^\otimes \simeq (\calS_*)^{\wedge}$$
where $(\Sp^{\geq 0})^\otimes$ is the symmetric monoidal $\infty$-category of connective spectra with smash product and $(\calS_*)^{\wedge}$ is the symmetric monoidal $\infty$-category of pointed spaces with smash product.

\begin{thm}\label{addequiv}
For every $\infty$-operad $\calO^\otimes$ in $\Op^*$ the $\infty$-operad $\calO_*^\otimes$ is pointed. If $\calO^\otimes$ is pointed then the functor $\calO^\otimes_* \to \calO^\otimes$ is an equivalence of $\infty$-operads.

For every $\infty$-operad in $\Op^\Pi$ the $\infty$-operad $\CGrp(\calO)^\otimes$ is additive and the $\infty$-operad $\CMon(\calO)^\otimes$ is preadditive.
If $\calO^\otimes$ is preadditive then the functor $\CMon(\calO)^\otimes \to \calO^\otimes$ is an equivalence of $\infty$-operads.
If $\calO^\otimes$ is additive then the inclusion $\CGrp(\calO)^\otimes \to \CMon(\calO)^\otimes$ is an equivalence of $\infty$-operads.
\end{thm}
\begin{proof}
The proof works as the proof of Theorem \ref{proppointad}. The seemingly different case that $\CGrp(\calO)^\otimes \to \CMon(\calO)^\otimes$ is an equivalence works the same: we import the result that it is an underlying equivalence from
\cite{GGN} and then argue that its an operadic equivalence using Proposition \ref{pressymerquicdk} and the left adjoint. Alternatively we show that $\CGrp(\calO)^\otimes \to \calO^\otimes$ is an equivalence and use 2-out-of-3.
\end{proof}

This immediately implies that the endofunctors $\CMon, \CGrp: \Op^\Pi \to \Op^\Pi$ and $(-)_*: \Op^* \to \Op^*$ are colocalization onto the full subcategories of preadditive and 
additive resp. pointed $\infty$-operads. Finally we get the analogue of \ref{corimo}.

\begin{corollary}\label{coradditivesmafwjq}
For every pointed $\infty$-operad $\calO^\otimes$ the induced functor
$$
\Alg^*_\calO\left(\calS_*\right) \to \Alg^*_\calO \left(\calS\right)
$$
is an equivalence of $\infty$-categories. For every preadditive $\infty$-operad $\calO^\otimes$ the induced functor
$$
\Alg^\Pi_\calO\left(\CMon(\calS)\right) \rightarrow \Alg^\Pi_\calO \left(\calS\right)
$$
is an equivalence of $\infty$-categories. For every additive $\infty$-operad $\calO^\otimes$ the induced functor
$$
\Alg^\Pi_\calO\left(\Sp^{\geq 0}\right) \xrightarrow{\Omega^\infty} \Alg^\Pi_\calO \left(\calS\right)
$$
is an equivalence of $\infty$-categories. 
\end{corollary}

%%%%%%%%%%%%%%%%%%%%%%%%%%%%%%%%%%%%%%%%%%%%%%%%%%%%%%%%%%%%%%%%%%%%%%
\section{Yoneda embeddings and lax monoidal structures}\label{secyoneda}
%%%%%%%%%%%%%%%%%%%%%%%%%%%%%%%%%%%%%%%%%%%%%%%%%%%%%%%%%%%%%%%%%%%%%%

In this section we want to discuss symmetric monoidal versions of Yoneda's lemma that we obtain in the stable and the additive setting. This also works with the necessary changes in the preadditive and pointed setting, but for 
the sake of readability we restrict to these two cases. First recall from \cite[Section 1.4]{HA} that 
for  a stable $\infty$-category $\calC$ the opposite is also stable and the functor 
\begin{equation}\label{equiv1}
\Fun^{\Lex}(\calC^{op}, \Sp) \xto{\Omega^\infty} \Fun^{\Lex}(\calC^{op}, \calS)
\end{equation}
is an equivalence. Here $\calS$ denotes the $\infty$-category of spaces and $\Sp = \Sp(\calS)$ denotes the $\infty$-category of spectra. For $\calC$ additive it is shown in \cite{GGN} that the functor 
\begin{equation}\label{twoee}
\Fun^{\Pi}(\calC^{op}, \Sp^{\geq 0}) \to \Fun(\calC^{op}, \calS) 
\end{equation}
is an an equivalence where $\Sp^{\geq 0}$ denotes the $\infty$-category of connective spectra. These statements are non-operadic analogues of Corollary \ref{corimo} and Corollary \ref{coradditivesmafwjq} and have already
been used extensively in the last chapters.

\begin{definition}\label{defrep}
Let $\calC$ be a stable $\infty$-category. Then a functor $F: \calC^{op} \to \Sp$ is called representable if $F$ preserves finite limits and if $\Omega^\infty F :\calC^{op} \to \calS$ is
representable in the classical sense. The object $c \in \calC$ representing $\Omega^\infty F$ is called the representing object for $F$. 
%Dually a functor $\calC \to \Sp$ is called corepresentable if it is representable as a 
%functor $\left(\calC^{op}\right)^{op} \to \Sp$.
Let $\calC$ be an additive $\infty$-category. Then a functor $F: \calC^{op} \to \Sp^{\geq 0}$ is called representable if $F$ preserves finite products and if $\Omega^\infty F :\calC^{op} \to \calS$ is
representable.  
%Dually a functor $\calC \to \Sp^{\geq 0}$ is called corepresentable if it is representable as a 
%functor $\left(\calC^{op}\right)^{op} \to \Sp^{\geq 0}$.
\end{definition}

Recall that the Yoneda embedding defines a functor $j: \calC \to \Fun(\calC^{op},\calS)$ which is an equivalence onto the full subcategory $\Fun^{\rep}(\calC^{op}, \calS) \subseteq \Fun(\calC^{op},\calS)$ of representable
functors. Directly from \eqref{equiv1} we get that for $\calC$ stable  post-composition with $\Omega^\infty$   induces an equivalence 
$
\Fun^{\rep}(\calC^{op}, \Sp) \to \Fun^{\rep}(\calC^{op}, \calS) 
$
where $\Fun^{\rep}(\calC^{op}, \Sp) \subseteq \Fun(\calC^{op}, \Sp)$ is the full subcategory on the representable functors. 
As a consequence we get a fully faithful inclusion $j_{\St}: \calC \to \Fun(\calC^{op}, \Sp)$ which we call the stable Yoneda embedding. In particular we get for every object $c \in \calC$ a representable functor represented by $c$. 
We will denote this functor as $\map_\Sp(-,c)$ and call for $d \in \calC$ the value $\map_\Sp(d,c)$ the mapping spectrum from $d$ to $c$. 
Similar for $\calC$ additive we get using \eqref{twoee} that the functors 
$$
\xymatrix{
\calC \ar[r]^-j & \Fun^{\rep}(\calC^{op}, \calS) &  \Fun^{\rep}(\calC^{op}, \Sp^{\geq 0}) \ar[l]_-{~~ \Omega^\infty}
}
$$
are equivalences. Composing them with the inclusion induces an additive Yoneda embedding $j_\Add: \calC \to \Fun(\calC^{op}, \Sp^{\geq 0})$. We denote the representable functor by $\map_{\Sp^{\geq 0}}(-,c)$ and call 
$\map_{\Sp^{\geq 0}}(d,c)$ the connective mapping spectrum. 

\begin{remark}
For a representable functor $F: \calC^{op} \to \Sp$ (or $\calC^{op} \to \Sp^{\geq 0}$) it follows a posteriori that it preserves all limits and not only finite limits (resp. products). This can be seen using that the analogous statement is true for the respective functor to spaces and limits of spectra can be tested
on the underlying spaces of all shifts. In fact if $\calC$ is presentable and stable (or presentable and additive) then a functor  $F: \calC^{op} \to \Sp$ is representable precisely if it preserves all limits. To see this use the 
equivalences
$$\calC \simeq \Sp(\calC) \simeq \calC \otimes \Sp \simeq \Fun^{\lim}(\calC^{op},\Sp)$$
which follows as explained in \cite[Section 4]{GGN}. 
\end{remark}

\begin{proposition}[Stable/Additive Yoneda lemma]\label{yoneda}
Let $\calC$ be a stable $\infty$-category and $F,G: \calC^{op} \to \Sp$ be functors such that $F$ is representable by $c \in \calC$ and $G$ preserves finite limits. Then we have an equivalence
$$
\Map_{\Fun(\calC^{op},\Sp)}(F,G) \simeq \Omega^\infty G(c)
$$
which is natural in $G$. 
If $\calC$ is additive and $F,G: \calC^{op} \to \Sp^{\geq 0}$ are functors where $F$ is represented by $c \in \calC$ and $G$ preserves finite products then we get a similar equivalence
$$
\Map_{\Fun(\calC^{op},\Sp^{\geq 0})}(F,G) \simeq \Omega^\infty G(c)
$$
\end{proposition}
\begin{proof}
This follows from \eqref{equiv1} and \eqref{twoee} together with the space valued Yoneda Lemma for $\infty$-categories.
\end{proof}

\begin{remark}
In the setting of Proposition \ref{yoneda} the $\infty$-catgory $\Fun(\calC^{op},\Sp)$ is stable, in particular also admits  mapping spectra. 
One can improve the statement of the proposition to the stronger statement that we have an equivalence of spectra 
$$\map_{\Sp}(F,G) \simeq G(c)$$
To this end consider the  functor $\map_{\Sp}(F,-): \Fun(\calC^{op},\Sp) \to \Sp$. It preserves finite limits 
and the underlying functor  $\Omega^\infty \map_{\Sp}(F,-)$ is given by $\Map(F,-)$. 
But the functor $\mathrm{ev}_c: \Fun(\calC^{op},\Sp) \to \Sp$ which takes $G$ to $G(c)$ also preserves finite limits and according to Proposition \ref{yoneda} 
it has the same underlying functor. Thus the two are equivalent. A similar argument in the additive case shows that $\map_{\Sp^{\geq 0}}(F,G) \simeq G(c)$.
\end{remark}

Now finally we want to discuss symmetric monoidal variants of Yoneda's Lemma. Recall from Corollary \ref{refyoneda} that for a symmetric monoidal $\infty$-category $\calC^\otimes$ the Yoneda embedding 
$j: \calC \to \Fun(\calC^{op}, \calS)$ admits a refinement 
$j': \CAlg(\calC) \to \Fun_{\lax}(\calC^{op}, \calS)$ which we call the symmetric monoidal Yoneda embedding.
Here $\Fun_{\lax}(\calC^{op}, \calS)$ denotes the $\infty$-category of lax symmetric monoidal functors. 
These are by definition operad maps from $(\calC^{op})^\otimes$ with  opposite symmetric monoidal structure to $\calS^\times$. We say that a lax symmetric monoidal functor from $\calC^{op}$
to either of the symmetric monoidal $\infty$-categories of spaces, connective spectra or spectra is representable if the underlying functor is representable in the sense of Definition \ref{defrep}. 

\begin{prop}
For  a symmetric monoidal $\infty$-category $\calC^\otimes$ the functor
$
j': \CAlg(\calC) \to \Fun^{\rep}_{\lax}(\calC^{op}, \calS)
$
is an equivalence  where $\Fun^{\rep}_{\lax}(\calC^{op}, \calS) \subseteq \Fun_\lax(\calC^{op}, \calS)$ is the full subcategory consisting of the representable functors. If $\calC$ is a stably symmetric monoidal $\infty$-category then the functor
$$
\Omega^\infty: \Fun^{\rep}_{\lax}(\calC^{op}, \Sp) \to \Fun^{\rep}_{\lax}(\calC^{op}, \calS) 
$$
induced by post-composition with the operad map $\Omega^\infty: \Sp \to \calS$ is an equivalences of $\infty$-categories. For $\calC$ additively symmetric monoidal the corresponding functor
$$
\Omega^\infty: \Fun^{\rep}_{\lax}(\calC^{op}, \Sp^{\geq 0}) \to \Fun^{\rep}_{\lax}(\calC^{op}, \calS)
$$
induced by post-composition with the operad map $\Omega^\infty: \Sp^{\geq 0} \to \calS$ is an equivalence.
\end{prop}
\begin{proof}
The first equivalence is immediate from Corollary \ref{refyoneda}. The second and third follow from Corollary \ref{corimo} and Corollary \ref{coradditivesmafwjq}.
\end{proof}

As a consequence we get a stable, symmetric monoidal Yoneda embedding $j_{\St}': \CAlg(\calC) \to \Fun_\lax(\calC^{op}, \Sp)$ for $\calC$ stably symmetric monoidal 
and an additive, symmetric monoidal Yoneda embedding $j_{\Add}': \CAlg(\calC) \to \Fun_\lax(\calC^{op}, \Sp^{\geq 0})$ for $\calC$ additively symmetric monoidal. As a first immediate, but certainly well known consequence we get the following result, where a commutative coalgebra 
in $\calC$ is by definition a commutative algebra in $\calC^{op}$.

\begin{corollary}
Let $\calC^\otimes$ be stably symmetric monoidal (additively symmetric monoidal). If $c \in \calC$ is a commutative coalgebra and $a \in \calC$ is a commutative algebra then the mapping spectrum $\map_{\Sp}(c,a)$ 
(resp. $\map_{\Sp^{\geq 0}}(c,a)$) admits the structure of an $\mathbb{E}_\infty$-ring spectrum. 
This structure is natural in coalgebra maps in $c$ and algebra maps in $a$.
\end{corollary}

\begin{corollary}
Let $\calC^\otimes$ be stably symmetric monoidal and $F: \calC^{op} \to \Sp$ a representable functor  or let $\calC^\otimes$ be additive and 
$F: \calC^{op} \to \Sp^{\geq 0}$ be representable. Denote the representing object by $c \in \calC$. Then there is a homotopy equivalence between the space of lax symmetric monoidal structures on $F$ and the space of 
commutative algebra structures on $c$.  
\end{corollary}
\begin{proof}
The space  $\CAlg(c)$ of commutative algebra structures on $c \in \calC$ is defined as the pullback
$$
\xymatrix{
\CAlg(c) \ar[d]\ar[r] & \CAlg(\calC) \ar[d] \\
\Delta^0 \ar[r]^{c} & \calC
}
$$
in $\Cat$ where the right vertical functor is the forgetful functor. This functor reflects equivalences, which implies that $\CAlg(c)$ is an $\infty$-groupoid. The space of lax symmetric monoidal structures on $F: 
\calC^{op} \to \Sp$ is similarly the fibre of $\Fun_\lax(\calC^{op},\Sp) \to \Fun(\calC^{op},\Sp)$ over $F$.
Now consider the following commuting diagram
$$
\xymatrix{
\CAlg(\calC) \ar[d]\ar[r]^-{j'} & \Fun^{\rep}_\lax(\calC^{op}, \calS) \ar[d] & \Fun^{\rep}_\lax(\calC^{op}, \Sp) \ar[l]_-{\Omega^\infty}\ar[r]\ar[d] & \Fun_\lax(\calC^{op}, \Sp)\ar[d] \\
\calC \ar[r]^-j & \Fun^\rep(\calC^{op},\calS) & \Fun^{\rep}(\calC^{op}, \Sp) \ar[l]_-{\Omega^\infty}\ar[r] & \Fun(\calC^{op}, \Sp)
}
$$
in which the vertical morphisms are the forgetful functors. The right hand square is a pullback by definition. The other two are pullbacks since all occurring horizontal morphisms are equivalences. 
Thus by pasting of pullbacks the fibres are homotopy equivalent.
% $$
% is a pullback diagram. By pasting together the two pullback diagrams we conclude that we have an equivalence
% $$
% \CAlg_\bC(c) \simeq  \CAlg_{\mathcal{P}(\bC)}(c)
% $$
% i.e.~that the space of algebra structures on $c$ is equivalent to the space of algebra structures on $\Map(-,c)$. But the latter is equivalent to the space of lax monoidal structures on that functor 
% (as stated in Proposition \ref{prop_glas}) which 
% finishes the proof.
\end{proof}

\begin{corollary}\label{cor}
Let $\calC$ be a symmetric monoidal $\infty$-category with unit object $\mathbbm{1} \in \calC$. 
\begin{enumerate}
\item
The functor $\Map(-,\mathbbm{1}): \calC^{op} \to \calS$ admits a canonical lax symmetric monoidal refinement. With this refinement it 
is initial in $\Fun_\lax(\calC^{op}, \calS)$. 
\item
If $\calC$ is additively symmetric monoidal then the functor $\map_{\Sp^{\geq 0}}(-,\mathbbm{1}): \calC^{op} \to \Sp^{\geq 0}$ admits a canonical lax symmetric monoidal 
refinement and with this refinement it 
is initial in $\Fun_\lax^\Pi(\calC^{op}, \Sp^{\geq 0})$.
\item
If $\calC$ is stably symmetric monoidal then the functor $\map_{\Sp}(-,\mathbbm{1}): \calC^{op} \to \Sp$ admits a canonical lax symmetric monoidal refinement and with this refinement it 
is initial in $\Fun_\lax^\Lex(\calC^{op}, \Sp)$. 
\end{enumerate}
\end{corollary}
\begin{proof}
In view of the equivalences $\Fun^\Lex_\lax(\calC^{op}, \Sp) \to \Fun^\Lex(\calC^{op}, \calS)$ for $\calC$ stable and $\Fun^\Pi_\lax(\calC^{op}, \Sp^{\geq 0}) \to \Fun^\Pi(\calC^{op}, \calS)$ for $\calC$ additive it suffices to prove the first claim. 

Since the Yoneda embedding $\calC \to \Fun(\calC^{op}, \calS)$ admits a symmetric monoidal structure (Proposition \ref{glassymyoneda}) it sends the tensor unit $\mathbbm{1}$ to the tensor unit. Thus the functor $\map(-,\mathbbm{1})$ is the tensor unit in $\Fun(\calC^{op},\calS)^\otimes$. 
But the tensor unit in
any symmetric monoidal $\infty$-category admits a canonical commutative algebra structure that makes it initial in the category of commutative algebras \cite[Setion 3.2]{HA}. This  together with the equivalence 
$\CAlg(\Fun(\calC^{op},\calS)) \simeq \Fun_\lax(\calC^{op}, \calS)$ from Proposition \ref{propopeee} implies the claim.
\end{proof}

This last corollary will be important in a future paper of the author with Markus Land where it is used to construct, among other things, an $\mathbb{E}_\infty$-map of ring spectra from $\mathrm{ku}$ to $l\mathbb{C}$. Here $ku$ is connective, complex $k$-theory and 
$l\mathbb{C}$ is the connective, symmetric $L$-theory spectrum of the complex numbers with complex conjugation as involution. The next corollary will be used in joint work of the author with Peter Scholze about the cyclotomic trace. 
\begin{corollary}\label{corstableyoned}
\begin{enumerate}
\item
The identity functor $\Sp \to \Sp$ is initial in $\Fun_\lax^\Lex(\Sp,\Sp)$. 
\item
The functor $\Omega^\infty: \Sp \to \calS$ is initial in $\Fun_\lax(\Sp, \calS)$.% and thus also in $\Fun_\lax^\Lex(\Sp,\calS) \subseteq \Fun_\lax(\Sp, \calS)$. 
\item
The identity functor $\calS \to \calS$ is initial in $\Fun_\lax(\calS,\calS)$.% and thus also in $\Fun_\lax^\Lex(\calS,\calS) \subseteq \Fun_\lax(\calS, \calS)$.
\item
The
 functor $\Sigma^\infty_+: \calS \to \Sp$ is initial in $\Fun_\lax(\calS, \Sp)$.
\end{enumerate}
\end{corollary}
\begin{proof}
The first two statements follow immediately by applying Corollary \ref{cor} to the dual category $\Sp^\mathrm{op}$ since 
$$\id_\Sp(-) \simeq \map_\Sp(\mathbbm{1}, -) \qquad \text{and} \qquad \Omega^\infty(-) \simeq \Map(\mathbbm{1}, -) 
$$
For the third statement we apply the first part of Corollary \ref{cor} to the dual of the category of spaces and use that $\id_\calS(-) \simeq \Map(\mathbbm{1},-)$. 
For the last statement   we use that the functor $\Sigma^\infty_+$ is the tensor unit in the Day convolution symmetric monoidal $\infty$-category 
$\Fun(\calS, \Sp)^\otimes$. This follows since the functor 
$$
\Fun(\calS, \calS)^\otimes \to \Fun(\calS, \Sp)^\otimes
$$
given by post-compositon with $\Sigma^\infty_+$ is symmetric monoidal according to Corollary \ref{morphism_co} and since the identity is the unit in the source category. Finally we use Proposition \ref{propopeee} to deduce that $\CAlg(\Fun(\calS,\Sp)) \simeq \Fun_\lax(\calS, \Sp)$. Then the statement is implied by the fact that the tensor unit is the initial $\mathbb{E}_\infty$-algebra.
\end{proof}

\begin{remark}
Note that  Corollary \ref{corstableyoned} is not true for an arbitrary stably symmetric monoidal $\infty$-category, in particular not for the category of chain complexes. It should be considered as a  fundamental property of the symmetric monoidal $\infty$-category of spectra. There is however an  analogous statement for chain complexes but it is about $H\mathbb{Z}$-linear functors (e.g. dg-functors) and not just exact functors. But since being $H\mathbb{Z}$-linear for functors is extra structure and not only a property the statement is not as close as useful as the one here.
\end{remark}
%Note that these last results can also be obtained slightly easier using presentable technology and thus do not need the full power of the results of this paper.

\begin{remark}
One can get versions of the results of this section for other operads $\calO^\otimes$ than $\mathbb{E}_\infty$, e.g. $\calO^\otimes = \mathbb{E}_n^\otimes$. In this case the $\infty$-category of lax symmetric monoidal functors has to be replaced by the category of $\calO$-symmetric monoidal functors which are $\calO$-algebras in the Day convolution structure on the functor category. The theory and the proofs in this case work mutatis mutandis the same but we will not need this extra generality. \end{remark}

\bibliographystyle{amsalpha}
\bibliography{Additive}

\end{document}